\documentclass[12pt,a4paper,reqno]{amsart}
\usepackage{preemble}

\title[Mixing of the Mineyev flow, orbital counting and Poincar\'e series for strongly hyperbolic metrics]
      {Mixing of the Mineyev flow, orbital counting and Poincar\'e series for strongly hyperbolic metrics}
\date{\today}

\author{Stephen Cantrell}
\address{Department of Mathematics, 
University of Warwick,
Coventry, CV4 7AL, UK}
\email{stephen.cantrell@warwick.ac.uk}

\date{\today. 
\\
2020 \textit{Mathematics Subject Classification}. Primary 20F67; Secondary 37D35, 37D40. 
\\
\textit{Key words and phrases}.
Orbital counting, Green metric, Poincar\'e series}

\title[Counting for strongly hyperbolic metrics]{Mixing of the Mineyev flow, orbital counting and Poincar\'e series for strongly hyperbolic metrics}

\begin{document}
\maketitle

\begin{abstract}
We obtain orbital counting results for the class of strongly hyperbolic metrics on hyperbolic groups. To achieve this we combine ergodic theoretic techniques involving the Mineyev topological flow and symbolic dynamics. Our results apply to the Green metric associated to an admissible,  finitely supported, symmetric random walk, to the Mineyev hat metric and to Hilbert length functions associated to Anosov representations. We also describe the domain of analyticity for the Poincar\'e series associated to these metrics, prove mixing results for the Mineyev topological flow and obtain correlation asymptotics for pairs of metrics. 
\end{abstract}

\section{Introduction}

Consider a group $\G$ acting by isometries on a metric space $(X,d)$ with fixed base point $o \in X$. The orbital counting problem is to ascertain the growth rate of
\[
\#\{ x \in \G: d(o,x\cdot o) < T\} \ \ \text{as $T \to \infty$. }
\]
A result of Roblin \cite{Roblin} states that, when $\G$ acts properly discontinuously and $(X,d)$ is a connected, simply connected Riemannian manifold with negative sectional curvatures such that
\begin{enumerate}
\item the length spectrum of $d$ is non-arithmetic; and,
\item the unit tangent bundle of the quotient space $X/\G$ admits a finite Bowen-Margulis-Sullivan measure,
\end{enumerate}
then there exist $C,\delta > 0$ such that
\[ 
 C e^{-\delta T} \ \#\{ x \in \G: d(o,x\cdot o) < T\} \to 1 \ \ \text{ as $T \to\infty$}.
\]
In fact, Roblin's result applies in more general $\text{CAT}(-1)$ settings. To prove this result Roblin exploits the mixing properties of the geodesic flow on $X/\G$ which are guaranteed thanks to the assumption that the length spectrum of $d$ is non-arithmetic. An alternative way to obtain the above counting result  in certain more restricted settings is to use ideas from thermodynamic formalism: if both $\G$ and the action of $\G$ on $(X,d)$ are sufficiently nice (cocompact, etc.), then the metric $d$ can be encoded as a potential on a subshift of finite type coming from a combinatorial coding of the group $\G$. This allows one to use techniques form thermodynamic formalism to study the analytic properties of the Poincar\'e series
\begin{equation*} 
\sum_{x \in \G} e^{-s d(o,x \cdot o)} \ \ \text{ for $s \in \C$},
\end{equation*}
associated to the metric $d$. The orbital counting asymptotic can then be deduced from the analytic properties of this series via a Tauberian theorem. This method can be used to study specific examples in which the group $\G$ is known to admit a combinatorial coding which exhibits certain good properties. This is the case for free groups and Fuchsian groups (i.e. groups that act cocompactly on the hyperbolic plane) as well as a few other examples. 

Although the result of Roblin mentioned above has many interesting applications, there are still many metrics for which the orbital counting problem remains unsolved. 
For example, we could equip a hyperbolic group $\G$ with the Green metric $d$ coming from a finitely supported, symmetric random walk (See Section 1.1). In this case $(\G,d)$ is not necessarily a $\text{CAT}(-1)$ metric space: $d$ may not be geodesic, but roughly geodesic (see Section 2.1).
It is natural to ask whether it is still possible to obtain orbital counting results in this setting. The Green metric in this case is an example of a \textit{strongly hyperbolic} metric: a class of metrics with nice asymptotic properties that was introduced by Nica and \v{S}pakula \cite{NicaSpakula} and inspired by the work of Mineyev \cite{MineyevFlow}. The strongly hyperbolic property can be seen as a generalisation of the $\text{CAT}(-1)$ property that can be satisfied by metrics that are roughly geodesic. In fact, $\text{CAT}(-1)$ metric spaces are strongly hyperbolic \cite{NicaSpakula}. This discussion leads us to question whether it is possible to resolve the orbital counting problem for the class of strongly hyperbolic metrics. Furthermore, can we descibe the domain of analyticity for the Poincar\'e series associated to such metrics?
These are the aims of this current work.
To prove our results we draw upon ideas from both of the techniques mentioned above: we use the Mineyev topological flow on the group and also thermodynamic techniques on a subshift of finite type coming from the Cannon coding. This will allow us to study Poincar\'e series and consequently deduce our counting results. We now state our main results and then explain our methods. Given a non-elementary hyperbolic group $\G$ we will write $\Dc_\G$ for the collection of $\G$-invariant hyperbolic metrics that are quasi-isometric to a word metric. We will write $\conj$ for the collection of conjugacy classes of $\G$. Given $d \in \Dc_\G$ we use the notation
\[
\ell_d[x] = \lim_{n\to\infty} \frac{d(o,x^n)}{n} \ \ \text{(where $[x]$ is the conjugacy class containing $x$)}
\]
for the translation distance function and $\conj'$ for the set of conjugacy classes with strictly positive translation length (i.e. non-torsion classes). We say that $d$ has non-arithmetic length spectrum if $\{\ell_d[x] : [x] \in \conj\}$ is not contained in $a \Z$ for some $a \in \R$. For functions $f,g: \R \to \R$ we write $f(T) \sim g(T)$ as $T\to\infty$ if $f(T)/g(T) \to 1$ as $T \to \infty$. Our orbital counting result is as follows. Recall that a number $\beta \in \R$ is called badly approximable if there exist $\alpha > 2, C>0$ such that $|\beta - \frac{p}{q}| \ge \frac{C}{q^\alpha}$ for all $p,q \in \Z$ with $q > 0$.

\begin{theorem}\label{thm.1}
Let $\G$  be a non-elementary hyperbolic group with identity element $o \in \G$. Suppose that $d \in \Dc_\G$ is a strongly hyperbolic metric such that the length spectrum of $d$ is not arithmetic. Then there exist $C>0$, $\delta >0$ such that
\[
\#\{x \in \G : d(o,x) < T \} \sim C  e^{\delta T}
\]
as $T \to \infty$. If there exist two conjugacy classes $[x], [y] \in \conj'$ such that the quotient $\ell_d[x]/\ell_d[y]$ is badly approximable then there exists $\k> 0 $ such that
\[
\#\{x \in \G : d(o,x) < T \}  = C  e^{\delta T}(1 + O(T^{-\k}))
\]
as $T \to \infty$. 
\end{theorem}
\begin{remark}
Note that if there exist $[x], [y] \in \conj'$ such that the quotient of their $d$-translation lengths is badly approximable, then the length spectrum of $d$ is necessarily non-arithmetic.
\end{remark}
The first part of this theorem is an immediate corollary of the following result regarding the analyticity of Poincar\'e series.
\begin{theorem}\label{thm.2}
Suppose that $d \in \Dc_\G$ is a strongly hyperbolic metric that has  non-arithmetic length spectrum and define the Poinar\'e series
\[
\eta(s) = \sum_{x \in \G} e^{-sd(o,x)}.
\]
Let $v_d$ be the exponential growth rate of $\#\{x \in \G: d(o,x) < T\}$. Then $\eta$ admits an extension that is analytic on $\mathrm{Re}(s)  \ge v_d$ except for a simple pole with positive residue at $s= v_d$.
\end{theorem}

Checking whether the length spectrum of a metric is non-arithmetic is usually difficult. However, we have the following result, essentially due to Gou\"ezel, Math\'eus and Macourant \cite[Proposition 3.2]{GMM2018}, that provides many examples. 

\begin{theorem}\cite{GMM2018} \label{thm.ls}
Suppose that $\G$ is a non-elementary hyperbolic group that is not virtually free. Then, every strongly hyperbolic metric $d \in \Dc_\G$ has non-arithmetic length spectrum.
\end{theorem}
We will prove this result in Section \ref{sec.prelim}. 
\begin{remark}
Note that free groups that act convex cocompactly by isometries on the hyperbolic plane provide examples of strongly hyperbolic metrics that have non-arithmetic length spectrum. In particular, virtually free groups can still admit strongly hyperbolic metrics that satisfy our theorems.
\end{remark}
We now outline the methods we use to prove these results. It has long been know that hyperbolic groups are susceptible to analysis through the use of ergodic theory and in particular symbolic dynamics. This is thanks to the work of Cannon \cite{Cannon} and Ghys and de la Harpe \cite{GhysdelaHarpe}: hyperbolic groups and their generating sets admit a combinatorial coding that allows for them to, in some sense, be encoded or represented by a dynamical system, specifically a subshift of finite type (see Section 2.4 for a precise explanation of this). This subshift of finite type arises from a finite directed graph for which the paths in this graph correspond to group elements. Unfortunately Cannon's construction does not in general shed any light on the connectedness properties of this graph (and it is possible to find Cannon codings that are not connected). Understanding these properties is crucial when one wants to apply techniques from ergodic theory. Indeed, the connectedness of the graph corresponds to the mixing properties of the corresponding subshift of finite type. In general, given a hyperbolic group $\G$ and finite  generating set $S$ we know that there exists a finite directed graph that encodes the pair $\G, S$ and that this directed graph $\mathcal{G}$ has finitely many connected components, i.e. we can decompose the vertex set of $\mathcal{G}$ into a finite collection of subsets such that if two vertices are in the same subset then there is a loop in $\mathcal{G}$ containing both of these vertices. Each of these connected components gives rise to a subshift of finite type and so we can represent a pair $\G,S$ (where $\G$ is a hyperbolic group and $S$ a finite generating set) by a finite collection of dynamical systems $\Sigma_1, \ldots, \Sigma_m$. We can apply results from thermodynamic formalism to each of these components individually. For example, given  H\"older continuous functions $r_j: \Sigma_j \to \R$ on each $\Sigma_j$ for $j=1, \ldots, m$ we obtain analytic pressure curves $\text{P}_j(-sr_j)$ for $s \in \R$  coming from a variational principle (see Section 2.5). A priori there is no reason for the thermodynamic data/properties of these curves  to coincide. However, a beautiful argument of Calegari and Fujiwara \cite{CalegariFujiwara2010} exploiting the ergodicity of certain Patterson-Sullivan measures on the boundary of the group, shows that, when each $r_1, \ldots, r_m$ correspond to certain functions on the group $\G$, their pressure curves $\text{P}_j(-sr_j)$  have the same first and second derivatives at $0$. 
 This result allows Calegari and Fujiwara and subsequently others to apply techniques from thermodynamic formalism to study hyperbolic groups despite the possible lack of structure for the Cannon graph. For example Gou\"ezel has applied (a generalised version of) this argument to obtain local limit theorems for the transition probabilities associated to random walks on hyperbolic groups \cite{GouezelLocalLimit}. In fact Gou\"ezel  has constructed a generalised thermodynamic framework that applies to subshifts and potentials coming from the study of hyperbolic groups. Other works that utilise the Calegari-Fujiwara argument include \cite{gouezel-lalley}, \cite{GTT2}, \cite{stats}, \cite{cantrell.sert}. We note that an important step in proving our theorems will be to improve upon this argument: we show that in fact all of the derivatives (not just the first and second derivatives) agree, i.e. the pressure functions are identical. 

The ideas and techniques developed in the aforementioned works can be used to compare the pressure curves $\text{P}_j(-sr_j)$ for $s \in \R$. Unfortunately, the techniques that have been developed so far can not be applied to study complex extensions of these curves. That is, they can not be used to compare the pressures $\text{P}_j(-sr_j)$  as $s$ varies in $\C$; particularly when $s$ is far from the real axis. To understand the domain of analyticity for our Poincar\'e series in Theorem \ref{thm.2} we need to overcome this issue. Roughly speaking the Calegari-Fujiwara argument allows us to characterise when the potentials $r_j$ on each $\Sigma_j$ are essentially constant functions (constant up to a coboundary). To prove results we need to understand more subtle properties of the functions $r_j$; we need to know when each $r_j$ on $\Sigma_j$ takes values in a  lattice $a\Z$ for some $a\in \R$. We therefore need to develope a better understanding of the structure of the Cannon graph. To do this we use ideas from \cite{cantrell.tanaka.2} which rely on the Mineyev topological flow: an analogue of a geodesic flow in our coarse geometric setting. To prove our results we take the following steps.
\begin{enumerate}
\item We use the Mineyev topological flow and the characterisation of flow invariant measures from \cite{cantrell.tanaka.2} to prove important structural results for the Cannon coding. This step relies on exploiting the ergodicity of the Mineyev flow.
\item Following on from the previous point, we improve upon the Calegari-Fujiwara argument mentioned earlier in the introduction: we show that  the derivatives (of every order) of certain pressure functions coincide. 
\item Using a combinatorial argument we show that, within certain components of the Cannon coding, it is possible to see all except the torsion conjugacy classes in the group as periodic orbits or loops.
\item Combining the previous two points, we show that having non-arithmetic length spectrum has important consequences for the cohomological properties of potentials on the coding.
\end{enumerate}
\begin{remark}
Proposition \ref{prop.main} is the key structural result mentioned in (1) above. This surprising result essentially allows us to work with the Cannon coding as if it were a single connected component. This helps us to overcome multiple issues that have previously prevented the full  application of thermodynamic formalism and symbolic dynamics to the shift space coming from the Cannon coding.
This proposition will likely have many other applications to the ergodic theoretic study of hyperbolic groups and geometries.
\end{remark}
Once we have completed the above steps, we can follow, along with a few technical alterations, the thermodynamic proof mentioned above to deduce Theorem \ref{thm.1} and Theorem \ref{thm.2}. 

Another way to tackle orbital counting problems is to use the method employed by Roblin in \cite{Roblin} which was mentioned at the beginning of the introduction. This method relies on exploiting the mixing properties of a geodesic flow. It may be possible to prove Theorem \ref{thm.1} in a similar fashion by using the Mineyev topological flow. However, we did not attempt to prove it in this way, as it is not obvious how to show Theorem \ref{thm.2} using this method. Although this is the case, the ideas used to prove Theorem \ref{thm.1} and Theorem \ref{thm.2} can be used to prove mixing results for the Mineyev topological flow. In the final section of this article we prove that the Mineyev flow is weak mixing under suitable (very mild) assumptions, see Section \ref{subsec.Mineyev}.

We now briefly outline the organisation of this article. In the first section we present examples and applications of our results. We then discuss some preliminary material related to hyperbolic groups, the Mineyev flow and the Cannon coding. After this, in Section \ref{section.codingstructure} we prove important structural results for the Cannon coding using ideas involving the Mineyev flow. In the following section, Section \ref{section.proof} we deduce Theorems \ref{thm.1} and Theorem \ref{thm.2}. We prove the rest of our results in the final section.


\subsection*{Acknowledgements} 
The author would like to thank Ryokichi Tanaka, Caleb Dilsavor and Eduardo Oreg\'on-Reyes for stimulating conversations and suggestions regarding this work. We also thank Richard Sharp for helpful comments and for explaning his work with R. Schwartz \cite{schwarz.sharp}.

We now present some applications of our results. Throughout the rest of this section, $\G$ will be assumed to by a non-elementary hyperbolic group.

\subsection{Random walks and the Green metric}
Our first application is to random walks.
Let $p$ be a probability measure on $\G$ that is finitely supported. We say that $p$ is admissible if the support of $p$ generates $\G$ as a semi-group. Suppose that $p$ is a finitely supported, symmetric and admissible probability measure on $\G$. If we define the Green function
\[
G_p(x,y) = \sum_{n=0}^\infty p^{\ast n} (x^{-1}y) \ \ \text{ for } \ x,y\in \G
\]
where $p^{\ast n}$ denotes the $n$-fold convolution of $p$, then the Green metric $d_p$ is defined as
\[
d_p(x,y) = - \log \left( \frac{G_p(x,y)}{G_p(o,o)} \right) \ \  \text{ for $x,y \in \G$ and where $o \in \G$ is the identity}.
\]
It is a consequence of the Ancona inequalities \cite{GouezelAncona} that this metric is strongly hyperbolic. In particular Theorem \ref{thm.1} and Theorem \ref{thm.ls} apply and we obtain:

\begin{corollary}\label{thm.GM}
Suppose that $d_p \in \Dc_\G$ is the Green metric associated to an admissible, finitely supported and symmetric random walk on a hyperbolic group $\G$. Assume further that $\G$ is not virtually free. Then, there exist $C >0$ such that
\[
\#\{x \in \G: d_p(o,x) < T\} \sim C e^{T}
\]
as $T \to \infty$. If there exist $[x], [y] \in \conj'$ such that $\ell_d[x]/\ell_d[y]$ is badly approximable then there exists $\k >0$ such that
\[
\#\{x \in \G: d_p(o,x) < T\} = C e^{ T}(1+O(T^{-\k}))
\]
as $T\to\infty$.
\end{corollary}
This result complements a result of Gou\"ezel \cite[Theorem 1]{GouezelAncona} (see also \cite{gouezel-lalley}) that provides the asymptotics growth rate of the transition probabilities $p^{\ast n}(x^{-1}y)$ as $n\to\infty$ for pairs of elements $x,y\in\G$.
Theorem \ref{thm.2} similarly applies and we obtain a result on the domain of analyticity for the Poincar\'e series for the Green metric.
This same discussion can be applied to the Mineyev hat metric $\widehat{d}$ which was constructed in \cite{MineyevFlow}. We obtain counting results and results regarding the Poincar\'e series for this metric. Note also, that as discussed in the introduction, we may replace the assumption that $\G$ is virtually free with the assumption that the length spectrum of $d$ is non-arithmetic.

\subsection{Mixing of the Mineyev flow} \label{subsec.Mineyev}
Our next application is regarding the Mineyev topological flow.
\begin{theorem} \label{thm.mixing}
Let $d$ be a strongly hyperbolic metric in $\Dc_\G$ and suppose that the length spectrum of $d$ is non-arithmetic. Then the Mineyev topological flow constructed using $d$ can be coded by a suspension flow $\textnormal{Sus}(\overline{\Sigma}_0, r_0)$ over a transitive subshift of finite type $\overline{\Sigma}_0$ and furthermore, this suspension flow is weak mixing. In particular, the Mineyev topological flow constructed using a strongly hyperbolic metric with non-arithmetic length spectrum is weak mixing.
\end{theorem}
\begin{remark}
i) The new content in this theorem is that the suspension flow is weak mixing when $d$ has non-arithmetic length spectrum. Indeed, the existence of a suspension flow coding for general strongly hyperbolic metrics was shown by Cantrell and Tanaka in \cite{cantrell.tanaka.2}. 

\noindent ii) By Theorem \ref{thm.ls} a Mineyev flow on a non-virtually free hyperbolic group is necessarily weak mixing.

\noindent iii) When there exist two conjugacy classes for which the quotient of the corresponding translation distances is badly approximable, our method shows that the Mineyev flow is rapid mixing.
\end{remark}
The terminology used in this result will be explained in Section \ref{section.mixing}.

\subsection{Correlation numbers}
As a corollary of our methods we can prove the following result that provides an asymptotic comparison between pairs of metrics.
\begin{theorem}\label{thm.cor}
Suppose that $d,d_\ast \in \Dc_\G$ are both strongly hyperbolic with exponential growth rate $1$ and that they are independent, i.e. if $a,b \in \R$ satisfy
\[
a\ell_d[x] + b\ell_{d_\ast}[x] \in \Z
\]
for all $x\in\G$, then $a = b =0.$ Then for any fixed $\epsilon >0$ there exist constants $ C > 0, 0 <\alpha < 1$ such that
\[
\#\{ x \in \G: d(o,x) \le T , \ |d_\ast(o,x) - d(o,x)| \le \epsilon   \} \sim \frac{Ce^{\alpha T}}{\sqrt{T}}
\]
as $T\to\infty$.
\end{theorem}
\begin{remark}
The assumption that both $d,d_\ast$ have exponential growth rate $1$ is a scaling condition to guarantee that one metric is not strictly smaller than the other, i.e. $\epsilon d_\ast(o,x) - C \le d(o,x) \le \epsilon'd_\ast(o,x) + C$ for  $x\in\G$ where $0<\epsilon < \epsilon' < 1$. If this were the case, the asymptotic expression in Theorem \ref{thm.cor} could not hold for $d,d_\ast$.
\end{remark}
This result is in the same spirit as a result of Schwartz and Sharp \cite{schwarz.sharp} on the correlation of pairs of hyperbolic metrics on a surface. The difference is that, in the above result we count over group elements with restrictions on the displacement functions $d,d_\ast$ where as Schwartz and Sharp count over primitive conjugacy classes with restrictions on the translation distances $\ell_d, \ell_{d_\ast}$. This difference can be seen in the asymptotic growth rates: in our result we have $T^{-1/2}$ in the asymptotic where as in \cite{schwarz.sharp} there is a $T^{-3/2}$ term. We suspect that Schwartz and Sharp's result generalises to our setting however we are unsure of how to prove this. In the setting considered in \cite{schwarz.sharp} the authors have access to a symbolic coding in which conjugacy classes in the group are in bijection with periodic orbits in the corresponding subshift of finite type. One of the key steps in this current work is to prove a similar, but weaker result, Proposition \ref{prop.loops} that allows us to see conjugacy classes as loops in a Cannon coding. We are unsure to what extent this proposition can be improved. However we note that torsion conjugacy classes can not be represented by loops and so we will never be able to form a bijection as described for surface groups.

\begin{remark}\label{rem.hp}
(1) Suppose that $V$ is a closed surface and $\mathfrak{g}_1$ and $\mathfrak{g}_2$ are two non-isometric hyperbolic metrics on $V$. Then, if $d_1,d_2 \in \Dc_\G$ are the metrics corresponding to the lifts of $\mathfrak{g}_1, \mathfrak{g}_2$ to $\G = \pi_1(V)$ then $d_1, d_2$ satisfy the conditions of Theorem \ref{thm.cor} by \cite{schwarz.sharp} (or Proposition \ref{prop.ls2} below). The same result is true for negatively curved metrics on $V$ by a result of Dal'bo \cite{Dalbo}.

\noindent (2) Theorem \ref{thm.cor} can be used to compare hyperbolic structures on different closed manifolds that share the same fundamental group. For example we could compare two metrics on a surface group $\G$, one coming from considering $\G$ as a fuchsian group acting on the hyperbolic plane and the other coming from a quasi-fuchsian embedding of $\G$ into the isometry group of hyperbolic $3$-space. This is an advantage of using the general coarse geometric set-up opposed to relying on an object which is intrinsic to the geometric structures under consideration, i.e. the geodesic flow on a manifold.
\end{remark}

In general, checking if two metrics are independent is difficult.  We will prove the following result that allows us to verify the independence property in certain cases. We say that two metrics $d,d_\ast \in \Dc_\G$ are \textit{roughly similar} if and only if there exists $\tau > 0$ such that $|d(o,x) - \tau d_\ast(o,x)|$ is uniformly bounded for $x \in \G$.
\begin{proposition}\label{prop.ls2}
Suppose that $\G$ has connected boundary. Then, two strongly hyperbolic metrics $d,d_\ast \in \Dc_\G$ are independent if and only if they are not roughly similar.
\end{proposition}

We prove this result after the proof of Theorem \ref{thm.cor} in the final section of the article. Note that if two metrics are roughly similar then they are dependent (i.e. not independent). Example of one-ended hyperbolic groups include surface groups and more generally the fundamental groups of compact hyperbolic manifolds. For example Theorem \ref{thm.cor} applies to a pair of Green metrics on a surface group as long as they are not roughly similar.

\begin{remark} \label{rem.na}
See Theorem 1.1 and 1.2 of \cite{cantrell.tanaka.1} for a list of conditions that are equivalent to rough similarity. For example $d,d_\ast \in \Dc_\G$ are roughly similar if and only if they have proportional marked length spectra. In particular if $\G$ has connected boundary and $d,d_\ast \in \Dc_\G$ are both strongly hyperbolic with exponential growth rate $1$, then if there is $[x] \in \conj'$ such that $\ell_{d_\ast}[x] \neq \ell_d[x]$, then $d,d_\ast$ satisfy the hypotheses of Theorem  \ref{thm.cor}.
\end{remark}

\subsection{Counting with error terms and Anosov representations}
We can also apply our results to (convex) cocompact actions on $\text{CAT}(-1)$ metric spaces and Anosov representations to obtain refinements of existing orbital counting results.
\begin{theorem}
Suppose that a group $\G$ acts properly discontinuously, (convex) cocompactly and freely by isometries on a $\textnormal{CAT}(-1)$ metric space $(X,d)$. Suppose $o \in X$ is a fixed base point. Suppose that there exist two conjugacy classes in $\conj'$ such that the quotient of their $d_X$ translation lengths is badly approximable. Then there exist $\delta, C, \k >0$ such that
\[
\#\{ x \in \G : d_X(x \cdot o, o) < T \} = Ce^{\delta T}(1 + O(T^{-\k}))
\]
as $T\to\infty$.
\end{theorem}

\begin{remark}
This result is a direct corollary of Theorem \ref{thm.1}: the metric $d_X$ lifts to a strongly hyperbolic metric in $\Dc_\G$. This result extends Rohblin's result mentioned in the introduction.
\end{remark}

Similarly we obtain refinements of orbital counting results for Anosov representations.  Let  $\rho: \G \to \SL_d(\R)$ be a representation of a finitely generated group. We say that $\rho$ is $1$-dominated if, when $\G$ is equipped with a finite generating set, the (quotient of the) first and second singular values of $\rho(x)$ separate exponentially quickly as the word length of $x$ increases. See \cite{BochiPotrieSambarino} for an introduction to  dominated representations. If a group $\G$ admits a dominated representation then it is necessarily hyperbolic \cite{BochiPotrieSambarino}. Also, being $1$-dominated is equivalent to being projective Anosov (as first introduced by Labourie for surface groups and extended to all hyperbolic groups by Guichard and Wienhard). It follows from \cite[Section 3.3]{cantrell.eduardo} and the work in \cite{cantrell.tanaka.2} that the Hilbert length functional 
\[
d_H(x,y) = \log\|\rho(x^{-1}y)\| + \log\|\rho(xy^{-1})\|  \ \ \text{ for } \ \ x,y \in \G
\]
associated to a projective Anosov representation $\rho: \G \to \SL_d(\R)$ is a strongly hyperbolic metric belonging to $\Dc_\G$. Hence we obtain the following.

\begin{theorem}
Suppose that  $\rho: \G \to \SL_d(\R)$ is a projetive Anosov representation and write $d_H$ for the corresponding Hilbert length functional. Write $o\in \G$ for the identity element. Suppose that there exist two conjugacy classes in $\conj'$ such that the quotient of their $d_H$ translation lengths is badly approximable. Then there exist $\delta, C, \k >0$ such that
\[
\#\{ x \in \G : d_H(o,x) < T \} = Ce^{\delta T}(1 + O(T^{-\k}))
\]
as $T\to\infty$.
\end{theorem}

\begin{remark}
This result is the first orbital counting result with an  error terms for the Hilbert length functional for general Anosov representations. This provides error terms in the counting results of \cite{sambarino} and \cite{Carvajales} for the Hilbert length functional.
The method presented in this work can likely be applied to obtain the same result for  the logarithm of the norm function $x \mapsto \log\|\rho(x)\|$ associated to a projective Anosov representation (as well as other linear functionals on the Cartan algebra). We do not pursue this here as to apply our methods we would need to verify that various results from \cite{cantrell.tanaka.2} that are stated for symmetric distances $d\in\Dc_\G$ also hold for the asymmetric distance $d(x,y) = \log\|\rho(x^{-1}y)\|$.
\end{remark}


\section{Preliminaries} \label{sec.prelim}

\subsection{Hyperbolic groups and metrics}

We assume that the reader is familiar with hyperbolic groups and metrics and so will only briefly introduce them here. See Section 2 of \cite{cantrell.tanaka.1} for a more comprehensive account.

For a metric space $(X, d)$ the Gromov product is defined by
\[
 (x|y)_w =\frac{1}{2}( d(w, x)+d(w, y)-d(x, y)) \ \ \text{for $x, y, w \in X$}.
\]
We say that $(X, d)$ is $\d$-{\it hyperbolic} for some $\delta \ge 0$ if
\begin{equation}
(x|y)_w \ge \min\left\{(x|z)_w, (y|z)_w\right\}-\d \quad \text{for all $x, y, z, w \in X$}.
\end{equation}
A metric space is called {\it hyperbolic} if it is $\d$-hyperbolic for some $\d \ge 0$.

A hyperbolic group is a finitely generated group $\G$ such that, when equipped with a word metric $|\cdot|_S$ associated to a finite generating set $S$,  becomes a hyperbolic metric space. Throughout this work all hyperbolic groups will be assumed to be non-elementary, i.e. they do not contain a finite index cyclic subgroup.  We will use the notation $\Cay(\G,S)$ to denote the Cayley graph of $\G$ with respect to $S$.

We say that two metrics $d,d_\ast$ on $\G$ are {\it quasi-isometric} if there exist constants $L > 0$ and $C \ge 0$ such that
\[
L^{-1} \, d(x, y)-C \le d_\ast(x, y) \le L \, d(x, y)+C \ \text{ for all $x, y \in \G$}.
\]
Throughout this work,  $\Dc_\G$ will denote the set of metrics which are left-invariant, hyperbolic and quasi-isometric to some (equivalently, any) word metric in $\G$. For $d\in \Dc_\G$ we use the notation $v_d$ for the exponential growth rate of $d$, i.e.
\[
v_d = \lim_{T\to\infty} \frac{1}{T} \log \#\{ x\in\G: d(o,x) < T\} \ \text{(which is necessarily strictly positive)},
\]
and write $\ell_d$ for the translation distance function for $d$, i.e. $\ell_d[x] = \lim_{n\to\infty} d(o,x^n)/n$ where $[x]$ is the conjugacy class containing $x$. 
If two metrics $d,d_\ast \in \Dc_\G$ have the property that there exists $\tau, C >0$ such that $|\tau d(o,x) - d_\ast(o,x)| < C$ for all $x \in \G$ then we say that $d,d_\ast$ are \textit{roughly similar} (here $o\in\G$ is the identity). Two metrics $d,d_\ast \in \Dc_\G$ are said to be \textit{independent} if the only constants $a,b \in \R$ for which the function $a \ell_d[x] + b\ell_{d_\ast}[x]$ takes integer values for $x\in\G$ are $a = b = 0$.
Fix a metric $d$ in $\Dc_\G$.
Given an interval $I \subset \R$ and constants $L, C >0$ we say that a map $\g: I \to \G$ is an \textit{$(L, C)$-quasi-geodesic} if 
\[
L^{-1} \, |s-t|-C \le d(\g(s), \g(t)) \le L\, |s-t|+C \  \text{ for all $s, t \in I$},
\]
and a \textit{$C$-rough geodesic} if
\[
|s-t|-C \le d(\gamma(s), \gamma(t))\le |s-t|+C \ \text{ for all $s, t \in I$}.
\]
A $0$-rough geodesic is referred to as a geodesic.
A metric space $(\G, d)$ is called $C$-{\it roughly geodesic} if for each pair of elements $x, y \in \G$ we can find a $C$-rough geodesic joining $x$ to $y$. We say that $(\G,d)$ is \textit{roughly geodesic} if it is $C$-roughly geodesic for some $C \ge 0$.
A geodesic metric space is a $0$-roughly geodesic metric space. Every metric in $\Dc_\G$ is roughly geodesic \cite{BonkSchramm}.
Moreover the \textit{Morse lemma} holds for such metrics: if  $d \in\Dc_\G$ is $C_0$-roughly geodesic
then for every $(L, C)$-quasi-geodesic $\g$ in $(\G, d)$ there is a $C_0$-rough geodesic $\g_0$ such that $\g$ and $\g_0$ are within Hausdorff distance $C$ depending on $L, C, C_0, \d$, where $\d$ is a hyperbolicity constant for $d$.

Our results are concerning the following class of metrics.

\begin{definition}
A hyperbolic metric $d$ in $\G$ is called {\it strongly hyperbolic} if there exist positive constants $c, R_0 > 0$ such that for all $R \ge R_0$, and all $x, x', y, y' \in \G$,
if
$ d(x, y)-d(x, x')+d(x', y')-d(y, y') \ge R,$
then
\[
|d(x, y)-d(x', y)-d(x, y')+d(x', y')| \le e^{-c R}.
\]
\end{definition}

\begin{remark} 
Mineyev \cite{MineyevFlow} has shown that every hyperbolic group admits a strongly hyperbolic metric in $\Dc_\G$ (referred to as the \textit{Mineyev hat metric}). Other examples include: 

\noindent (1) The Green metric associated to random walk on $\G$ as discussed in the introduction.

\noindent (2) The orbit metric coming from a properly discontinuous, cocompact, free and isometric action on a $\text{CAT}(-1)$ metric space.

\noindent (3) Linear functions on the Cartan Algebra associated to Anosov representations.
\end{remark}

Hyperbolic groups can be compactified using their ideal boundary $\partial \G$ which consists of equivalence classes
of divergent sequences. A sequence of group elements $\{x_n\}_{n=0}^\infty$ diverges
if $(x_n|x_m)_o$ diverges as $\min\{n, m\}$ tends to infinity (here $o$ is the identity in $\G$). Two divergent sequences $\{x_n\}_{n=0}^\infty$ and $\{y_n\}_{n=0}^\infty$ are {\it equivalent} if 
$(x_n|y_m)_o$ diverges as $\min\{n, m\}$ tends to infinity. 
If $d \in \Dc_\G$ is $C$-roughly geodesic then for each $\x$ in $\partial \G$ there exists a $C$-rough geodesic $\g:[0, \infty) \to \G$ 
such that $\g(0)=o$ and $\g(n) \to \x$ as $n \to \infty$.
Furthermore for each pair of boundary elements $\x, \y$ in $\partial \G$ with $\x \neq \y$ there is a $C$-rough geodesic connecting $\x$ to $\y$, see \cite[Proposition 5.2]{BonkSchramm}.

The Gromov product extends to points in $\G \cup \partial \G$ for any $d\in\Dc_\G$. Fix $d\in\Dc_\G$ and
let
\[
(\xi|\eta)_o = \sup\left\{\liminf_{n \to \infty}(x_n|y_n)_o \ : \ \x=[\{x_n\}_{n=0}^\infty], \  \y=[\{y_n\}_{n=0}^\infty]\right\},
\]
for $\x, \y \in \G \cup \partial \G$,
where if $\x$ or $\y$ is in $\G$, then we take $\{x_n\}_{n=0}^\infty$ as the constant sequence $x_n=\x$ for all $n\ge 0$.
The Busemann function $ \beta_w(x, \x)$ associated to $d$ based at $w$ is given by
\[
\beta_w(x,\x) = \sup\left\{ \limsup_{n\to\infty} d(x,\x_n) - d(w,\xi_n) : \{ \x_n\}_{n=0}^\infty \right\}.
\]
When $d$ is strongly hyperbolic $\beta_w$ is obtained as a genuine limit 
\[
\b_w(x, \x)=\lim_{n \to \infty}\( d(x, x_n)- d(w, x_n)\)= d(w, x)-2( x | \x )_w \quad \text{for $(x, \x) \in \G \times \partial \G$},
\]
where $\x=[\{x_n\}_{n=0}^\infty]$,
and is continuous with respect to $\x$ in $\partial \G$.
Moreover in this case we have the cocycle identity
\[
\b_w(xy, \x)= \b_w(y, x^{-1}\x)+ \b_w(x w, \x) \quad \text{for $w, x, y \in \G$ and $\x \in \G\cup \partial \G$}.
\]
For general metrics $d\in\Dc_\G$ this identity holds up to a uniformly bounded additive error and in this case $\b_w$ is refered to as a \textit{quasi-cocycle}.

The quasi-metric in $\partial \G$ associated to  a metric $d \in \Dc_\G$ is defined by
\[
q(\x, \y) =\exp\(-(\x|\y)_o\) \ \  \text{for $\x, \y \in \partial \G$},
\]
where $q(\x, \x)=0$.
Recall that a quasi-metric satisfies $q(\x, \y)=0$ if and only if $\x=\y$, $q(\x, \y)=q(\y, \x)$, and  $q(\x, \y) \le e^{2\d}\max\left\{q(\x, \z), q(\z, \y)\right\}$ for $\x, \y, \z \in \partial \G$.
We endow $\partial \G$ with a topology by choosing $\{\y \in \partial \G \ : \ q(\x, \y)<\e\}$ for $\x \in \partial \G$ and $\e>0$ as an open basis.

\subsection{Patterson-Sullivan construction}

For $d\in\Dc_\G$ we define the {\it shadow sets}
\[
\Oc(x, R) =\{\x \in \partial \G \ : \ (\xi |x)_o \ge d(o, x)-R\} \ \  \text{for $x \in \G$ and $R\ge 0$}.
\]
Given a quasi-metric $q(\xi, \eta) = \exp(-(\xi|\eta)_o)$ on $\partial \G$ associated with a metric $d$ in $\Dc_\G$
we define
\[
B(\xi, r) =\{\y \in \partial \G \ : \ q(\x, \y) \le r\} \ \ \text{for $\x \in \partial \G$ and $r \ge 0$}.
\]
For a pair of metrics $d,d_\ast \in \Dc_\G$ and a real number $a \in \R$, we can find unique $b \in \R$ such that $ad + bd_\ast$ has exponent $1$, that is
\[
\sum_{x \in \G} e^{-s(ad(o,x) + bd_\ast(o,x))}
\]
has abscissa of convergence $1$ as $s$ varies (and $a, b$ are fixed). In fact $b$ can be obtained from the \textit{Manhattan curve} for the pair $(d,d_\ast)$, see Section \ref{section.codingstructure}.
Corollary 2.10 of \cite{cantrell.tanaka.1} shows the following.

\begin{proposition}[Corollary 2.10 \cite{cantrell.tanaka.1}] \label{prop.qcm} 
Let $d,d_\ast \in \Dc_\G$ be metrics and $a,b \in \R$ so that $ad + bd_\ast$ has exponent $1$. Then, there exists a Radon measure $\mu_{a,b}$ which is ergodic with respect to the action of $\G$ on $\partial \G$. There exists a constant $C_{a,b} > 1$ such that for any $x\in\G$
\[
C_{a,b}^{-1}e^{-a\beta_{o}(x,\xi) - b\b_{\ast,o}(x,\xi)} \le  \frac{dx_\ast\mu_{a,b}}{d\mu_{a,b}} \le C_{a,b}e^{-a\beta_{o}(x,\xi) - b\b_{\ast,o}(x,\xi)}
\]
where $\beta_{\ast o}$, $\b_o$ are the Busemann functions for $d_\ast, d$ respectively.
\end{proposition}

The measures in this proposition also satisfy the following: there exists $C>1$ depending on $a,b$ and $R$ such that
\begin{equation}\label{eq.shadow}
C^{-1} e^{-ad(o,x) - bd_\ast(o,x)} \le \mu_{a,b}(\Oc(x,R)) \le Ce^{-ad(o,x) - bd_\ast(o,x)}  \ \text{ for each $x \in \G$.}
\end{equation}

We end this subsection with the proof of Theorem \ref{thm.ls} from the introduction.

\begin{proof} [Proof of Threorem \ref{thm.ls}]
Let $\beta_o: \G \times \partial \G \to \R$ be the Busemann cocyle associated to a strongly hyperbolic metric $d \in \Dc_\G$. In this case $\beta_o$ is a H\"older cocycle (i.e. it is not a quasi-cocycle). Given a hyperbolic element $x\in\G$ we will use 
$x^+$ (respectively $x^-$) for the attracting (respectively repelling) fixed point in $\partial \G$. That is, $x^+$ and $x^-$ are given by $[\{ x^n\}_{n=0}^\infty]$ and $[\{ x^{-n}\}_{n=0}^\infty]$. It is easy to check from the cocycle property and hyperbolicity that if $x \in\G$ is hyperbolic (i.e. $\{ x^n: n \in \mathbb{Z}\}$ is infinite) then $\beta_o(x,x^+) = \ell_d[x]$ and $\beta_o(x,x^-) = -\ell_d[x]$. Now suppose that the length spectrum of $d$ is arithmetic. Then by Proposition 3.2 of \cite{GMM2018} there exists a hyperbolic element $x \in \G$ with $\beta_o(x,x^+) = \beta_o(x,x^-)$ where $x^-, x^+ \in \partial \G$ are the repelling and attracting fixed points of $x$. However $\beta_o(x,x^+) = \ell_d[x] > 0$ and $\beta_o(x,x^-) = - \ell_d[x] < 0$ resulting in a contradiction.
\end{proof}

\subsection{Invariant measures on the boundary square}\label{Sec:invariant}

The boundary square $\partial^2 \G = (\partial \G)^2\setminus \{\text{\rm diagonal}\}$ is endowed with the restriction of the product topology and the diagonal action of $\G$ on $\partial^2 \G$ is continuous.

Take two metrics $d,d_\ast \in \Dc_\G$. Given $a\in\R$ let $b\in\R$  be such that $ad + b d_\ast$ has exponent $1$. Let $\mu_{a,b}$ be the measure from Proposition \ref{prop.qcm} then we have the following.

\begin{proposition}[Proposition 2.8 and Example 2.9 \cite{cantrell.tanaka.2}] \label{prop.prodm}
For each $a \in \R$ let $b \in \R$ be so that $ad + bd_\ast$ has exponent $1$. Then there is a $\G$ invariant Radon measure $\Lambda_{a,b}$ on $\partial^2\G$ equivalent to
\[
\exp(2(\x,\eta)_{\ast o} + 2(\x,\eta)_o) \ \mu_{a,b} \otimes \mu_{a,b}
\]
and $\Lambda_{a,b}$ is ergodic with respect to the action of $\G$ on $\partial^2 \G$.
\end{proposition}

We will be interested in various quantities related to the product measures in this proposition. The next definition
and lemma introduce these quantities.

\begin{definition}
If $\L$ is a Radon measure on $\partial^2 \G$,
then we define
\[
D_q(\x_-, \x_+: \L) = \liminf_{r \to 0}\frac{\log \L(B(\x_-, r)\times B(\x_+, r))}{\log r} \ \ \text{for $(\x_-, \x_+) \in \partial^2 \G$}, 
\]
where balls $B(\x, r)$ are defined in terms of a quasi-metric $q$ in $\partial \G$.
\end{definition}

Given $d,d_\ast \in \Dc_\G$ and $\x \in \partial \G$ we define
\[
\tau_{\text{inf}}(\xi) = \liminf_{t\to\infty} \frac{d_\ast(\g_\x(o), \g_\x(t))}{d(\g_\x(o), \g_\x(t))} \ \text{ and } \ \tau_{\text{sup}}(\xi) = \limsup_{t\to\infty} \frac{d_\ast(\g_\x(o), \g_\x(t))}{d(\g_\x(o), \g_\x(t))}
\]
where $\g_\xi : [0,\infty) \to (\G,d)$ is any quasi-geodesic with $\x = \lim_{t\to\infty} \g_\x(t)$. Note that these quantities are well-defined by the Morse Lemma. If $\tau_\text{inf}(\x) = \tau_\text{sup}(\x)$ we denote their common value by $\tau(\x)$.

\begin{lemma}[Lemma 3.2 \cite{cantrell.tanaka.2}]\label{lem.intersect} 
For each $a,b \in\R$ so that $ad + bd_\ast$ has exponent $1$ we have that $\tau_\textnormal{inf}(\x) = \tau_\textnormal{sup}(\x)$ for $\mu_{a,b}$ almost every $\x \in \G$. Furthermore for such $a,b$ there exists a constant $\tau_{a,b}$ with $\tau(\xi) = \tau_{a,b}$
for $\mu_{a,b}$ almost every $\xi \in \partial \G$. 
\end{lemma}

\subsection{Markov structures: Cannon's coding}

Fix a finite generating set $S$ for $\G$.

\begin{definition} \label{def.sms}
Let  $\Ac=(\Gc, w, S)$ be a collection where
\begin{enumerate}
\item $S$ is a finite generating set for $\G$;
\item $\Gc=(V, E, \ast)$ is a finite directed graph with a vertex $\ast$ called the \textit{initial state}; and,
\item $w: E \to S$ is a labelling such that
for a directed edge path $(x_0, x_1, \dots, x_{n})$ (where $(x_i, x_{i+1})$ corresponds to a directed edge)
there is an associated path in the Cayley graph $\Cay(\G, S)$ beginning at the identity: the path corresponds to
\[
(o, w(x_0,x_1), w(x_0,x_1)w(x_1,x_2), \dots, w(x_0,x_1)\cdots w(x_{n-1}, x_n)).
\]
\end{enumerate}
Let $\ev$ denote the map that sends a finite path to the endpoint of the corresponding path in $\Cay(\G,S)$, i.e.
$\ev(x_0,\ldots, x_n) = w(x_0,x_1)\cdots w(x_{n-1},x_n)$.
We say that $\Ac$  is a strongly Markov structure for $\G, S$ if 
\begin{enumerate}
\item for each vertex $v\in V$ there exists a directed path from $\ast$ to $v$;
\item for each directed path in $\Gc$ the associated path in $\Cay(\G, S)$ is a geodesic; and,
\item the map $\ev$ defines a bijection between the set of directed paths from $\ast$ in $\Gc$ and $\G$.
\end{enumerate}
\end{definition}

Cannon first showed that certain cocompact Kleinian groups and generating sets admit strongly Markov structures \cite{Cannon}. Ghys and de la Harpe extended Cannons result and showed that every hyperbolic group and finite generating set admits a strongly Markov automatic structure (cf. \cite{GhysdelaHarpe},  \cite[Section 3.2]{Calegari}).

For technical reasons which will become apparent later, we augment the above strongly Markov structure by introducing an additional vertex labelled $0$ to $V$. We also add directed edges from every vertex $x \in V \cup \{0\} \backslash \{\ast\}$ to $0$ and define $\lambda(x,0) = o$ (the identity in $\Gamma$) for every $x \in V \cup \{0\} \backslash \{\ast\}$. We will assume that every strongly Markov structure has been augmented in this way and will abuse notation by labelling the augmented structure, its edge and vertex set by $\mathcal{G}$, $V$ and $E$ respectively. This graph $\mathcal{G}$ allows us to introduce a subshift of finite type.

Let $A$ be the $k \times k$ (where $k$ is the number of vertices in $\Gc$), $0-1$ transition matrix describing $\Gc$. 
We use the notation $A_{i,j}$ to denote the $(i,j)$th entry of $A$. The one-sided subshift of finite type associated to $A$ is the space
\[
\Sigma_A = \{(x_n)_{n=0}^{\infty} : x_n \in \{1,2,...,k\}, A_{x_n,x_{n+1}}=1, n \in \mathbb{Z}_{\ge 0}\}.
\]
When $A$ is clear, we will drop $A$ from the notation in the definition of the shift spaces and will simply write $\SS$. Given $x\in\SS_A$ we write $x_n$ for the $n$th coordinate of $x$. 
The shift map $\sigma: \Sigma_A \rightarrow \Sigma_A$ sends $x$ to $y= \sigma(x)$ where $y_n=x_{n+1}$ for all $n \in \mathbb{Z}_{\ge 0}$. For each $0<\theta <1$ we can define a metric $d_\theta$ on $\Sigma_A$. Let $x,y \in \Sigma_A$.  If $x_0 = y_0$ then we define $d_\theta(x,y)= \theta^N,$ where $N$ is the largest positive integer such that $x_i = y_i$ for all $0 \le i<N$. If $x_0 \neq y_0$ we set $d_\theta(x,y)=1$. We let
\[
F_\theta(\Sigma_A) = \{r:\Sigma_A \rightarrow \mathbb{C}:\text{$r$ is Lipschitz with respect to $d_\theta$}\}
\]
which we equip with the norm $\|r\|_\theta = |r|_\theta + |r|_\infty$ where $|r|_\infty$ is the sup-norm and $|r|_\theta$ denotes the least Lipschitz constant for $r$. This space is a Banach space. Two functions $r_1,r_2 \in F_\theta(\SS_A)$ are said to be cohomologous (denoted by $r_1 \sim r_2$) if there exists a continuous function $h:\Sigma_A \rightarrow \mathbb{C}$ such that $r_1=r_2+h\circ \sigma - h.$ A theorem of Livsic asserts that if $r\in F_\theta$ then the set $\{S_nr(x) - Cn: x\in \Sigma_A, n \in \mathbb{Z}_{\ge 0}\}$ is bounded if and only if $r$ is cohomologous to the constant function with value $C$. Here $S_nr(x) = r(x) + r(\sigma(x))+...+r(\sigma^{n-1}(x))$ denotes the $n$th Birkhoff sum. A function $r: \Sigma \to \R$ (on a  subshift of finite type $\Sigma$) is arithmetic if there exists $a \in \R$ such that
\[
\{S_nr(x) : x \in \Sigma \text{ and } \sigma^n(x) = x\} \subseteq a\Z.
\]
If no such $a$ exists we say that $r$ is non-arithmetic.

The two-sided subshift conisists of all bi-infinite paths in $\Gc$:
\[
\overline{\SS}_A = \{(x_n)_{n=-\infty}^{\infty} : x_n \in \{1,2,...,k\}, A_{x_n,x_{n+1}}=1, n \in \mathbb{Z}\}.
\]
When $A$ is clear, we will write $\overline{\SS}$. 
 We will always use the overlined notation for two-sided shifts and the non-overlined notation for one-sided shifts.
On $\overline{\SS}_A$ the shift $\sigma: \overline{\Sigma}_A \rightarrow \overline{\Sigma}_A$ sends $x$ to $y= \sigma(x)$ where $y_n=x_{n+1}$ for all $n \in \mathbb{Z}$.  As for one sided shifts for each $0<\theta <1$ we can define a metric $d_\theta$ on $\overline{\SS}_A$. Let $x,y \in \overline{\SS}_A$.  If $x_0 = y_0$ then we define $d_\theta(x,y)= \theta^N,$ where $N$ is the largest positive integer such that $x_i = y_i$ for all $0 \le |i| <N$. If $x_0 \neq y_0$ we set $d_\theta(x,y)=1$. We can define $F_\theta(\overline{\SS}_A)$ analogously to the one-sided case.

It will be convenient to extend the map $\ev$ defined on finite paths in Definition \ref{def.sms} to infinite paths.
If $x = (x_k)_{k=0}^\infty \in \SS$ and $n \in \Z_{\ge 0}$ then we set
\[
\ev_n(x) = \ev(x_0, x_1, \ldots, x_n).
\]
We also define $\ev(x)$ to be the point in $\partial \G$ corresponding to the infinite geodesic determined by $x$.
This map can be used to pushforward measures on subshifts to measures on $\partial \G$. We use the notation $\ev_\ast$ for the
corresponding pushforward map.

\subsection{Thermodynamic formalism}
We say that a directed graph $\Gc$ is connected if there exists a path between any pair of vertices in $\Gc$.
A \textit{connected component} of a finite directed graph is a maximal, connected subgraph.
Given a Cannon coding $\Gc$ associated to a pair $\G,$ $S$, since the $\ast$ state only has outgoing edges, $\Gc$ will never be connected. We can however decompose $\Gc$ into connected components. 

Given such a connected component $\Cc$, there exists a maximal integer $p_\Cc \ge 1$ such that the length of every closed loop in $\Cc$ has length divisible by $p_\Cc$. We call 
$p_\Cc$ the period of $\Cc$.
When $p_\Cc =1$ we say that the component $\Cc$ is aperiodic and the corresponding subshift $\SS_\Cc$ defined over $\Cc$ is mixing. In general since $\Cc$ is connected $(\SS_\Cc, \sigma)$ is {\it topologically transitive}. That is, for any two non-empty open sets $U$ and $V$, there exists $n \in \Z$ such that $U\cap \s^n V \neq \emptyset$.
In the case that $p_\Cc >1$ we can decompose the vertex set $V(\Cc)$ for $\Cc$ into a disjoint collection of vertices $V_1, \ldots, V_{p_\Cc}$ such that $V(\Cc) = \bigsqcup_{j \in \Z/p_\Cc\Z} V_j$. Letting $\SS_j$ for each $j=1,\ldots, p_\Cc$ denote the set of elements in $\SS_\Cc$ that contain sequences starting with a vertex in $V_j$ we have that $\sigma(\SS_j) = \SS_{j+1}$ where $j,j+1$ are taken modulo $p_\Cc$. Furthermore, each system $(\SS_j, \sigma^{p_\Cc})$ is a mixing subshift of finite type.

Using the same notation as before, let $(\overline{\SS}_\Cc, \sigma)$ be the two-sided shift space defined over $\Cc$.
For a function $\Psi$ on $\SS_\Cc$ we can define the induced potential $\Psi$ on $\overline{\SS}_\Cc$ by setting $\P(x)=\P(x_0, x_1, \dots)$ (we abuse notation and also call this function $\P$). This function only depends on the non-negative indices of an input. We now want to introduce the \textit{pressure} of a potential $\P$ on $\SS_\Cc$ and will do so via the so-called variational principle.
In the following we use  $\Mcc(\s, \SS_\Cc)$ to denote the set of $\sigma$-invariant Borel probability measures on $\SS_\Cc$.
We write $h(\s, \lambda)$ for the measure theoretical entropy of $(\SS_\Cc, \s, \lambda)$ (see Section 3 of \cite{ParryPollicott}).
Given a finite sequence $x_0, \ldots, x_{n-1}$ we write $[x_0, \dots, x_{n-1}]$ for the corresponding cylinder sets in $\overline{\SS}_\Cc$, i.e. $[x_0, \dots, x_{n-1}]$ denotes the collection of sequences in $\overline{\SS}_\Cc$ that agree with $(x_0,x_1,\ldots,x_{n-1})$ on the $0$th to $(n-1)$st indices. These objects are defined analogously for $\overline{\Sigma}_\Cc$.
\def\Pr{{\rm P}}
\begin{proposition}[The Variational Principle, Theorem 3.5 \cite{ParryPollicott}]\label{prop.vp}
If $\P$ is a H\"older continuous function on $\overline{\SS}_\Cc$ then the following supremum
\[
\Pr_\Cc(\P) =\sup_{\lambda \in \Mcc(\sigma, \overline{\SS}_\Cc)}\left\{h(\s, \lambda)+\int_{\overline{\SS}_\Cc}\P\,d\lambda\right\}
\]
is attained by a unique $\sigma$-invariant Borel probability measure $\m_\Cc$ on $\overline{\SS}_\Cc$, and further there exists a positive constant $c > 1$ such that
\begin{equation}\label{Eq:Gibbs}
c^{-1} \le \frac{\m_\Cc[x_0, \dots, x_{n-1}]}{\exp\(-n\textnormal{P}_\Cc(\P)+S_n \P(x)\)} \le c
\end{equation}
for all $x \in [x_1, \dots, x_{n-1}]$ and for all $n \in \Z_{\ge 1}$,
where $S_n\P$ is the $n$th Birkhoff sum.
The quantity $\Pr_\Cc(\P)$ is referred to as the \textit{pressure} of $\P$ over $\Cc$.
\end{proposition}

The pressures can be related to the spectral radii of the following linear operators, known as \textit{transfer operators}. Fix $\P \in F_\theta(\SS_\Cc)$ and define $L_\Cc : F_\theta(\SS_\Cc) \to F_\theta(\SS_\Cc)$ by
\[
L_{\Cc} \o(x) = \sum_{\s(y) = x} e^{\P(y)} \o(y).
\]
Then the spectral radius of this operator is $e^{\Pr_\Cc(\P)}$ and furthermore, this operator has $p_\Cc$ simple maximal eigenvalues: $e^{2\pi l/p_\Cc}e^{\Pr(\P)}$ for $l =1, \ldots, p_\Cc$. The rest of the spectrum is contained in a disk in $\C$ centered at $0$ with radius strictly smaller than $e^{\Pr(\P)}$ \cite{ParryPollicott}.

When we are considering a Cannon coding with multiple connected components and $\P \in F_\theta(\SS_A)$, we will write
\[
\Pr(\P)=\max_\Cc \Pr_\Cc(\P),
\]
where $\Cc$ runs over all components in $\Gc$.
\begin{definition}
A component $\Cc$ is called {\it maximal} for $\P$ (or \textit{$\P$-maximal}) if $\Pr(\P)=\Pr_\Cc(\P)$.
We call a component $\Cc$ \textit{word maximal} if it is maximal for the constant function with value $1$. 
\end{definition}
Note that the word maximal components for a Cannon coding for $\G,S$ are precisely the components
that have spectral radius given by the growth rate of $|\cdot|_S$.

\begin{lemma}[Example 4.11 \cite{cantrell.tanaka.1} and Lemma 4.8 \cite{cantrell.tanaka.1}]\label{lem.holder}
Take a strongly hyperbolic metric $d$ in $\Dc_\G$.
Then, for any Cannon coding and associated subshift $\SS$, we can find a H\"older continuous potential $\P_d $ such that
\[
S_n\P_d(x)=\sum_{i=0}^{n-1}\P_d(\s^i(x))= d(o, \ev_n(x))+O(1) \quad \text{for all $x \in \SS$}
\]
uniformly in $x, n$. If $x_0 = \ast$ (i.e. the first entry in $x$ corresponds to the state $\ast$) then $S_n\P_d(x) = d(o,\ev_n(x))$.  In fact, we may take $\P_d(x) = \beta_o(\ev_1(x),\ev(x))$.
Furthermore, if $v_d$ is the exponential growth rate of $d$, then $\Pr(-v_d\P_d) = 0$.
\end{lemma}

As we continue, given a strongly hyperbolic metric $d \in \Dc_\G$, $\P_d$ will refer to the H\"older potential from this lemma.
We end this section with the following important observation.

\begin{lemma}[Lemma 4.7 \cite{cantrell.tanaka.1}]\label{lem.disjointmc}
Fix a strongly hyperbolic metric $d\in\Dc_\G$, a generating set $S$ and a Cannon coding $\Gc$ for $\G, S$. Then the $-v_d \P_d$ maximal components are disjoint, that is, there does not exist a path in $\Gc$ connecting vertices belonging to distinct $-v_d \P_d$ maximal components.
\end{lemma}





\subsection{The Mineyev flow}
The Mineyev topological flow is an analogue of a geodesic flow in our coarse geometric setting.
To avoid introducing more notation we will not include an exposition defining the flow. Instead we refer the reader to
Sections 2.7 of \cite{cantrell.tanaka.1}.

Given a strongly hyperbolic metric $d\in\Dc_\G$ we can construct the Mineyev flow $\Fc_\k$ coming from a cocycle $\k$ associated to $d$. In \cite{cantrell.tanaka.2} it was shown that the Mineyev flow defined over a strongly hyperbolic metric can be coded by a suspension flow over a subshift of finite type. For more information of suspensions of subshifts in the current setting, see Section 4.2 of \cite{cantrell.tanaka.2}.

\begin{theorem}\cite{cantrell.tanaka.2} \label{thm.enhanced}
For any non-elementary hyperbolic group $\G$, 
let $\Fc_\k$ be the topological flow space associated with a cocycle $\k$ associated to a strongly hyperbolic metric.
There exists a topologically transitive subshift of finite type $(\overline{\SS}_0, \s)$ with a positive H\"older continuous function $r_0$
such that a natural continuous map $\Pi_0$ from the associated suspension flow $\text{Sus}(\overline{\SS}_0, r_0)$ to $\Fc_\k$
is surjective, equivariant with the flows and the cardinality of each fiber is uniformly bounded.
\end{theorem}

Importantly, the H\"older function $r_0$ in the above theorem  can be taken to be $S_N\Psi_d$  where $\P_d$ is the function from Lemma \ref{lem.holder} and $N \ge 1$ is a sufficiently large integer. Furthermore, the suspension flow $\text{Sus}(\overline{\SS}_0, r_0)$  can be taken as the suspension over $\overline{\SS}_\Cc$ where $\Cc$ is any $-v_d\Psi_d$ maximal component (see Section 4.4 of \cite{cantrell.tanaka.2}). This fact will be crucial to our work.

Using Theorem \ref{thm.enhanced} we can improve upon some of the results from the previous section. Fix a strongly hyperbolic metric $d$. Throughout the following let $\overline{\SS}_0$ denote the the subshift of finite type defined over a fixed $-v_d \Psi_d$ maximal component as in Theorem \ref{thm.enhanced}.  The following two lemmas from \cite{cantrell.tanaka.2} will be useful.

\begin{lemma}[Lemma 5.2 \cite{cantrell.tanaka.2}]\label{lem.measub}
Suppose that $\L$ is a Radon measure on $\partial^2 \G$. Then, there exists a probability measure $\lambda$ on $\overline{\SS}_0$ such that $\lambda$ is $\s$-invariant and  $
\ev_\ast \lambda \le C \L $ on $\partial^2 \G$
for some positive constant $C>0$.
If for the strongly hyperbolic metric $d$ defining $\Psi_d$, $D_{q}(\x_-, \x_+: \L) \ge D$ for $\L$-almost every $(\x_-, \x_+)$ in $\partial^2 \G$
where $q$ is the associated quasi-metric in $\partial \G$,
then
\[
h(\s, \lambda)+\frac{D}{2}\int_{\overline{\SS}_0}\P_d \ d\lambda \ge 0,
\]
where $\P_d(x)= \b_o(\ev_1(x), \ev(x))$ for $x \in \overline{\SS}_0$.
\end{lemma}

\begin{lemma}\label{lem.Birkhoff}[Lemma 5.3 \cite{cantrell.tanaka.2}]
Take $d_\ast \in \Dc_\G$ (and $d$ is still our fixed strongly hyperbolic metric) and write $\mu_{a,b}$ for the measure obtain from Proposition \ref{prop.prodm}.
Suppose there exists a function $\P_{a,b} \in F_\theta(\overline{\SS}_0)$ on $(\overline{\SS}_0, \s)$ such that
\[
S_n\P_{a,b}(x)=\sum_{i=0}^{n-1}\P_{a,b}(\s^i(\o))= ad(o, \ev_n(x)) + b d_\ast(o,\ev_n(x)) +O(1) 
\]
 uniformly in $n \ge 1, x \in \overline{\SS}_0$,
and a $\s$-invariant Borel probability measure $\lambda$ on $(\overline{\SS}_0, \s)$ 
satisfying $\ev_\ast \lambda \le C \mu_{a,b}$.
Then for the local intersection number $\tau_{a,b}$ of $ad+bd_\ast$ relative to $d_\ast$ for $\m_{a,b}$-almost every $\x \in \partial \G$, we have that
\[
\int_{\overline{\SS}_0}\P_{a,b} \ d\lambda=\tau_{a,b} \int_{\SS_0}\P_d \ d\lambda,
\]
where $\P_d(\o)=\b_o(\ev_1(x), \ev(x))$ for $x \in \overline{\SS}_0$.
\end{lemma}

\subsection{Growth tightness}
As seen in the previous sections, hyperbolic groups have strong combinatorial properties.
We will need to exploit an additional combinatorial property of hyperbolic groups to prove our results. This
property is known as growth quasi-tightness.

  For an element $w \in \G$ and a real number $\D \ge 0$,
  we say that a $S$-geodesic word \textit{$\D$-contains $w$} 
  if it contains a subword $h$ such that $h = h_1wh_2$
  for some $h_1, h_2 \in \Gamma$ with $|h_1|_S, |h_2|_S \le \D$.
  Let $Y_{w,\D}$ be the set of group elements $x \in \G$ such that $x$ is represented by some $S$-geodesic word which does not $\D$-contain $w$.
  
 \begin{definition}  
A group $\G$ is called growth quasi-tight if for any fixed generating set $S$
  there exists $\D_0> 0$ such that for any $w \in \Gamma$,
  \[
  \lim_{n\to \infty} \frac{\#(Y_{w,\D_0} \cap S_n)}{\#S_n}=0.
  \]
  Here $S_n$ is the $n$-sphere centered at the identity in $\Cay(\G,S)$.
  \end{definition}
  By   \cite[Theorem 3]{arzlys} hyperbolic groups are growth quasi-tight. An immediate consequence of this property is the following result that we will use repeatedly.
 
 \begin{lemma} \label{lem.gqtprod}
 Fix a word maximal component $\Cc$ in a Cannon coding for a hyperbolic group $\G$ and generating set $S$. Let $\G_\Cc$ denote the collection of group elements that correspond to a path in $\Cc$, i.e. the collection of elements in $\G$ that correspond to multiplying the edges labellings along a path in $\Cc$. Then there exists a finite set $F \subset \G$ such that $F \G_\Cc F = \G$.
 \end{lemma}
\begin{proof}
This was observed in the proof of Lemma 4.6 in \cite{GMM2018}. See Section 4.5 of \cite{cantrell.tanaka.1} for more details.
\end{proof}


\section{Structural results for the Cannon coding} \label{section.codingstructure}
The aim of this section is to use the Mineyev topological flow and combinatorial properties of hyperbolic groups to gain a better understanding of the structural properties of the Cannon coding. Recall that a component of a Cannon coding for $\G,S$ is word maximal if it is maximal for the constant function with value $1$. 
\subsection{Study of the components}
The following result is key to our work.
\begin{proposition} \label{prop.main}
Let $\G$ be a hyperbolic group equipped with finite symmetric generating set $S$ and fix a Cannon coding $\Sigma$ for $\G, S$. Suppose that $d \in \Dc_\G$ is a strongly hyperbolic metric with exponential growth rate $v_d$ and that $\Psi_d$ is the H\"older potential on $\Sigma$ corresponding to $d$.  Then, the $-v_d \Psi_d$ maximal components are precisely the word maximal components.
\end{proposition}

We will break the proof into two lemmas. 
To prove these results it will be convenient to introduce a new object: Manhattan curves. These are curves that encode information about pairs of metrics in $\Dc_\G$. They were introduced by Burger \cite{BurgerManhattan} in the setting of cocompact actions on rank one symmetric spaces and have proved a useful tool in obtaining rigidity results for metrics. Since their conception, Manhattan curves have been defined an studied in a variety of settings. In the current setting, Cantrell and Tanaka studied these curves in \cite{cantrell.tanaka.1} and \cite{cantrell.tanaka.2}.

\begin{definition}
Given $d, d_\ast \in \Dc_\G$ the Manhattan curve $\theta_{d/d_\ast} : \R \to \R$ is defined as follows: for each $a\in \R$, $\theta_{d/d_\ast}(a)$ is the abscissa of convergence of 
\[
\sum_{x\in\G} e^{-ad(o,x) - b d_\ast(o,x)}
\]
as $a$ remains fixed and $b$ varies.
\end{definition}

We will mainly be interested in the case that $d_\ast$ is a word metric. When this is the case and $d_\ast$ is the word metric associated to a finite generating set $S$ we will denote the Manhattan curve corresponding to $d, d_\ast$ by $\theta_{d/S}$.
Suppose we have fixed a finite generating set $S$ and a Cannon coding $\Sigma$ for $\G, S$. Given a strongly hyperbolic metric $d \in \Dc_\G$ with corresponding H\"older potential $\Psi_d$ on $\SS$ it follows from the work in Section 4.4 of \cite{cantrell.tanaka.1} that $\theta_{d/S}(t) = \max_{\Cc} \text{P}_\Cc(-t\Psi) = \text{P}(-t\Psi_d)$ where this maximum is taken over all components $\Cc$ coming from the fixed Cannon coding.

\begin{lemma} \label{lem.comp1}
Consider the Manhattan curve $\theta_{d/S}$ for a strongly hyperbolic metric $d \in \Dc_\G$ and word metric associated to $S$. Then we can realise $\theta_{d/S}$ as
\[
\theta_{d/S}(t) = \textnormal{P}_\mathcal{C}(-t\Psi_d) \ \ \text{ for all } \ t\in\R
\] 
where $\mathcal{C}$ is any fixed word maximal component.
\end{lemma}
To prove this lemma we use a combinatorial argument that exploits the fact that hyperbolic groups are growth quasi-tight. Our next lemma is the following.

\begin{lemma}\label{lem.comp2}
Consider the Manhattan curve $\theta_{d/S}$ for a strongly hyperbolic metric $d \in \Dc_\G$ and word metric associated to $S$. Then we can realise $\theta_{d/S}$ as
\[
\theta_{d/S}(t) = \textnormal{P}_{\mathcal{C}_0}(-t\Psi_d) \ \ \text{ for all } \ t\in\R
\] 
where $\mathcal{C}_0$ is any fixed component that is maximal for $-v_d\Psi_d$.
\end{lemma}
To prove this we use the Mineyev topological flow and a characterisation of flow invariant measures. Assuming these lemmas we can prove Proposition \ref{prop.main}.

\begin{proof} [proof of Proposition \ref{prop.main}]
Consider the Manhattan curve $\theta_{d/S}$ for the metric $d$ and word metric associated to $S$.  We know that this Manhattan curve is obtained as
\[
\theta_{d/S}(t) = \text{P}(-t\Psi_d) = \max_{C} \text{P}_C(-t\Psi_d)
\]
where the maximum is taken over all components. However, by our previous two lemmas, $\theta_{d/S}$ is also obtained as
\[
 \max_{C} \text{P}_C(-t\Psi_d)= \theta_{d/S}(t) = \text{P}_{\mathcal{C}_0}(-t \Psi_d) = \text{P}_\mathcal{C}(-t\Psi_d)
\]
where $\mathcal{C}_0$ is any maximal component for $-v_d\Psi_d$ and $\mathcal{C}$ is any word maximal component. In particular, each word maximal component must be a $-v_d\Psi_d$ maximal component and vice versa.
\end{proof}

We now need to complete the proofs of Lemmas \ref{lem.comp1} and \ref{lem.comp2}.

\begin{proof} [proof of Lemma \ref{lem.comp1}]
Suppose $\mathcal{C}$ is a word maximal component in the fixed Cannon coding for $\Gamma, S$. Let $p_\Cc$ be the period of $\mathcal{C}$ and recall that we can write
\[
\bigsqcup_{j=1}^{p_\Cc} \Sigma_{j} = \Sigma_\Cc
\]
where the $(\Sigma_j, \sigma^p)$ are mixing  subshifts and $\sigma(\Sigma_{j}) = \Sigma_{j+1}$ where $j, j+1$ are taken modulo $p_\Cc$. 
Consider $\Sigma_{1}$ and let $V_1$ denote the collection of start vertices for paths  in $\Sigma_{1}$.
We let $\Gamma^{V_1}$ denote the collection of elements
\[
\Gamma^{V_1} = \{ x \in \Gamma: \ \text{$x$ can be realised as a loop in $\mathcal{C}$ based at a vertex in $V_1$} \}.
\]
Note that the  $S$-word length of every element in $\Gamma^{V_1}$ is divisible by $p_\Cc$. We define a restricted Poincar\'e series
\[
 P^{V_1}(s,t) =  \sum_{x \in \Gamma^{V_1}} e^{-s|x|_S - td(o,x)}.
\]
Writing $S_n = \{g\in \Gamma: |g|_S =n\}$ and $\Gamma^{V_1}_n = \Gamma^{V_1} \cap S_n$ we note that
\[
\limsup_{n\to\infty} \frac{\#\Gamma^{V_1}_n}{\#S_n} > 0.
\]
In particular, there exists a finite set $F \subset \Gamma$ such that $F \Gamma^{V_1} F = \Gamma$. This follows from the same argument used to prove Lemma \ref{lem.gqtprod}.
Now fix $s \in \mathbb{R}$ and note that, if $t \in \mathbb{R}$ is such that $P(s,t) < \infty$, then
\begin{equation} \label{eq:1}
\infty > P(s,t) \ge P^{V_1}(s,t).
\end{equation}
Also note that, if $|g|_S = n$ then there exist $f_1,f_2 \in F$ and $\overline{g} \in \Gamma^{V_1}$ such that $g= f_1 \overline{g} f_2$. Since $|f_1|_S$ and $|f_2|_S$ are at most $C = \max_{f\in F}|f|_S$ we have that $n- 2C \le |\overline{g}|_S \le n+2C$. Further, since  $d(o,x)$ is comparable to $|x|_S$ there exists $\widetilde{C} > 0$ such that $|d(o,\overline{g})) -d(o,x)| < \widetilde{C}$. For each $g$ we can find $\overline{g}$ as above and the association $g \mapsto \overline{g}$ induces a map 
\[
\psi: S_n \to \bigcup_{k= -2C}^{2C} \Gamma^{V_1}_{n+k}
\]
which is at most $(\#F)^2$ to $1$. In particular
\[
\sum_{|g|_S = n} e^{-s|x|_S - t d(o,x)} \le (\#F)^2 \sum_{k=-2C}^{2C} \sum_{g \in \Gamma^{V_1}_{n+k}} e^{-s|x|_S - td(o,x)}  e^{\widetilde{C}(|s|+ |t|)}.
\]
Therefore,
\begin{equation} \label{eq:2}
\begin{split}
P(s,t) &\le \sum_{n=1}^\infty (\#F)^2 e^{\widetilde{C}(|s|+|t|)} \sum_{k= -2C}^{2C}
 \sum_{g\in \Gamma^{V_1}_{n+k}}  e^{-s|x|_S - td(o,x)}\\ 
 &=(\#F)^2e^{\widetilde{C}(|s|+|t|)} \sum_{n=1}^\infty \sum_{k=-2C}^{2C}  \sum_{g \in \Gamma^{V_1}_{n+k}}   e^{-s|x|_S - td(o,x)}\\
&\le (4 C e^{ \widetilde{C}(|s|+|t|)}(\#F)^2) \  P^{V_1}(s,t). 
\end{split}
\end{equation}
Combining equations (\ref{eq:1}) and (\ref{eq:2}) we deduce that for fixed $t \in \mathbb{R}$, the abscissa of convergence of $s \mapsto P(s,t)$ and $s\mapsto P^{V_1}(s,t)$ are the same.
It follows that the Manhattan curve $s = \theta_{d/S}(t)$  is the solution  to the equation $\text{P}_\mathcal{C}(-s - t\Psi_d)=0$ where $\Psi_d$ is the H\"older potentials for $d$ on the coding. To see this note that
\[
P^{V_1}(s,t) = \sum_{n \ge 1} \sum_{x \in \SS_1: \ \s^{pn}(x) = x} e^{-spn} e^{-tS_{pn}\P_d(x)}
\]
and by Proposition 5.1 \cite{ParryPollicott}, for fixed $t\in\R$, 
\[
\sum_{x \in \SS_1: \  \s^{pn}(x) = x} e^{-tS_{pn}\P_d(x)} = e^{pn\Pr_\Cc(-t\P_d)}(1+O(\theta^n))
\]
 for some $\theta \in (0,1)$ as $n\to\infty$.
Overall we have shown that
$\theta_{d/S}(t) = \text{P}_\mathcal{C}(-t\Psi_d)$ as required.
\end{proof}

We now move on to the proof of the second lemma.

\begin{proof} [proof of Lemma \ref{lem.comp2}]
We will imitate the proof of Theorem 5.4 from \cite{cantrell.tanaka.2}: the crucial difference being that the strongly hyperbolic metric used to define the Mineyev flow will be the same as the metric $d$ we are studying.
Fix a $-v_d\Psi_d$ maximal component $\Cc_0$ and define $\P_{s, t}:=-s-t \Psi_d$. Note that the Birkhoff sums of $\P_{s,t}$ output $-s|\cdot|_S - td(o,\cdot)$ up to a uniformly bounded error as in Lemma \ref{lem.Birkhoff}.
If $s=\th_{d/S}(t)$, then $-s|\cdot|_S - t d(o,\cdot)$ has exponent $0$ and thus we have $\text{P}(\P_{s, t})=0$.
Therefore we have that $\Pr_{\Cc_0}(\P_{s, t}) \le 0$.
If $s=\th_{d/S}(t)$, then there exists a $\G$-invariant Radon measure $\L_{s, t}$ on $\partial^2 \G$ equivalent to $\m_{s, t}\otimes \m_{s, t}$ (Proposition \ref{prop.prodm}).
For the measure $\L_{s, t}$ on $\partial^2 \G$, by Lemma \ref{lem.measub}
there is a $\sigma$-invariant probability measure $\lambda_{s, t}$ on $\overline{\SS}_0$ such that
\[
\ev_\ast \lambda_{s, t} \otimes \Leb_{[0, T]} \le c\L_{s, t}\otimes dt \ \ \text{on $\partial^2 \G \times \R$},
\]
for some positive constants $T, c>0$, and thus $\ev_\ast \lambda_{s, t} \le C \m_{s, t}\otimes \m_{s, t}$ on $\partial^2 \G$ for some positive constant $C>0$. Here $\Leb_{[0,T]}$ denotes the Lebesgue measure on $[0,T]$.
Suppose we used the metric $d$ to define $(\overline{\SS}_0, \s)$ in Theorem \ref{thm.enhanced} and that $q$ is the corresponding quasi-metric in $\partial \G$. Then
we have that
\[
D_{q}(\x_-, \x_+: \L_{s, t})=2\tau_{s, t} \ \  \text{for $\L_{s, t}$ almost all $(\x_-, \x_+) \in \partial^2 \G$},
\]
where $\tau_{s, t}$ is the local intersection number of $-s |\cdot|_S-t d(o,\cdot)$ relative to $ d$ , $\m_{s, t}$-almost everywhere.
This follows from Lemma \ref{lem.intersect} and equation (\ref{eq.shadow}).
Moreover, Lemma \ref{lem.Birkhoff} implies that
\[
\int_{\overline{\SS}_0}\P_{s, t} \ d\lambda_{s, t} = \tau_{s, t}\int_{\overline{\SS}_0}\P_d \ d\lambda_{s, t}
\ 
\text{
and so
} \ 
h(\s, \lambda_{s, t})+\int_{\overline{\SS}_0}\P_{s, t}\,d\lambda_{s,t} =h(\s, \lambda_{s, t})+\tau_{s, t}\int_{\overline{\SS}_0}\P_d\,d\lambda_{s, t} \ge 0.
\]
We then have by Proposition \ref{prop.vp} that $\text{P}_{\Cc_0}(\P_{s, t}) \ge 0$ and so
 for $t \in \R$, if $s=\th_{d/S}(t)$ then
$ -s + \text{P}_{\Cc_0}(-t \Psi_d) = \text{P}_{\Cc_0}(\P_{s, t})=0$
concluding the proof.
\end{proof}

\subsection{Conjugacy classes and loops}
The aim of this subsection is to prove that, within the Cannon coding, we can see the conjugacy classes of $\G$ as loops.
Our main result of this section is a corollary of the following proposition. Recall that the $S$-word length of a conjugacy class $[x] \in \conj$ is the $S$-length the shortest word(s) in $[x]$. Furthermore there is a uniform constant $C >0$ such that the word length and $S$-translation length of any conjugacy class agree up to error $C$.

\begin{proposition}\label{prop.loops}
Suppose that $\mathcal{C}$ is a word maximal component in a Cannon coding for $\G, S$. Then there exist $L, N >0$ such that for any conjugacy class $[g] \in \conj(\G)$ with word length at least $L$, there is a periodic orbit $x =(x_0, x_1, \ldots, x_l, x_0, \ldots ) \in \Sigma_\mathcal{C}$ (for some $l >0$) such that $\ev(x_0, \ldots, x_l,x_0)$ belongs to one of $[g^{\pm N}]$.
\end{proposition}

\begin{proof}
Fix $q \in \Z_{\ge 0}$ larger than the number of vertices in the Cannon graph and let $k = \#B(r) = \#\{x \in \G : |x|_S \le r\}$ where $r>0$ is to be chosen later. Given a conjugacy class $[g] \in \conj$ with minimal element $g$ with $|g|_S \ge L$, we will show that there exists $0 \le n \le qk$ and $a \in B(r)$ such that one of $a^{-1} g^{\pm n} a$ and hence one of $[g^{\pm n}]$ is represented by a loop in $\mathcal{C}$. The conclusion then follows with $N = (kq)!$.

Let $[g]$ be a conjugacy class of length at least $L$ with $g$ being an element of minimal length. Then, the path in $C_\G(S)$ obtained from repeating $g$ is an $L$-local geodesic (i.e. subpaths of length $L$ are geodesic). Further, assuming that $L$ is large enough this path is a $(K,C)$-quasigeodesic for some $K,C >0$ (see \cite{Calegari}). Now by Lemma \ref{lem.gqtprod} there is a finite path $y=(y_0,y_1, y_2, \ldots, y_w)$ (for some $w$) in $\mathcal{C}$ such that $\ev(y) = h_1 g^{qk} h_2$ for some $h_1, h_2 \in B(r_0)$ for some $r_0 >0$ independent of $g$. 
Now note that the path in $C_\G(S)$ (beginning at the identity) obtained from composing $h_1$ followed by $pq$ copies of $g$ and ending with  $h_2$ is a $(K',C')$-quasigeodesic for some $K',C' > 0$ independent of $g, q, k, h_1, h_2$ (we can enlarge $L$ if necessary). Fix $r>0$ so that each $(K',C')$-quasigeodesic is at most Hausdorff distance $r$ away from a genuine geodesic. Since the path in $C_\G(S)$ obtained from composing $h_1, g^{qk}, h_2$  has the same endpoints as $\ev(y)$ (which is a genuine geodesic), for each $0 \le n  \le kq$ there exists $p_n$ with 
\[
(h_1g^n)^{-1} \ev(y^{p_n}) \in B(r) \ \ \text{where \, $y^{p_n} = (y_1, \ldots,y_{p_n})$. }
\]
Consider the function
\[
f: \{0,\ldots, kq \} \to V \times B(r) \  \  \text{defined by} \ \  f(n) = (y_{p_n}, (h_1g^{n})^{-1} \ev(y^{p_n}))
\]
where $V$ is the vertex set of the Cannon graph. This is a function from a set of cardinality $ kq + 1$ to a set of cardinality $kp$. Hence there is $n_1 < n_2$  such that
\[
y_{p_{n_1}} = y_{p_{n_2}} \ \ \text{ and } \ \ (h_1g^{n_1})^{-1} \ev(y^{p_{n_1}}) = (h_1g^{n_2})^{-1}\ev(y^{p_{n_2}}) : =a.
\]
 Suppose $p_{n_1} < p_{n_2}$. Let $x = (y_{p_{n_1}}, \ldots, y_{p_{n_2}}, y_{p_{n_1}} \ldots ) \in \Sigma_\mathcal{C}$ and $n = n_2 - n_1$. Then $x$ is a periodic point with period $n$ and
\begin{align*}
a^{-1}g^n  a &= \ev(y^{p_{n_1}})^{-1} h_1 g^{n_1} g^{n_2 - n_1} g^{-n_2} (h_1)^{-1} \ev(y^{p_{n_2}}) \\
&= \ev(y^{p_{n_1}})^{-1} \ev(y^{p_{n_2}})\\
&= \ev(y_{p_{n_1}}, \ldots, y_{p_{n_2}}) = \ev(x_0, \ldots, x_{n-1}, x_0).
\end{align*}
Therefore $[g^n]$ is represented by a loop in $\mathcal{C}$. If $p_{n_1} \ge p_{n_2}$ then setting  $x = (y_{p_{n_2}}, \ldots, y_{p_{n_1}}, y_{p_{n_2}} \ldots )$ and $n=n_2 - n_1$ gives
\begin{align*}
a^{-1}g^{-n}  a &= \ev(y^{p_{n_2}})^{-1} h_1 g^{n_2} g^{-n_2 + n_1} g^{-n_1} (h_1)^{-1} \ev(y^{p_{n_1}}) \\
&= \ev(y^{p_{n_2}})^{-1} \ev(y^{p_{n_1}})\\
&= \ev(y_{p_{n_2}}, \ldots, y_{p_{n_1}}) = \ev(x_0, \ldots, x_{n-1}, x_0)
\end{align*}
which implies that $[g^{-n}]$ is represented by a loop in $\mathcal{C}$. This concludes the proof.
\end{proof}

\begin{corollary} \label{cor.loops}
Suppose that $\mathcal{C}$ is a word maximal component in a Cannon coding. Then there exists an integer $M \ge 1$ such that, for every non-torsion conjugacy class $[g] \in \conj'$, there exists a periodic orbit $(x_0, x_1, \ldots, x_l, x_0, \ldots)$ (for some $l >0$) such that $\ev(x_0,\ldots,x_l,x_0)$ belongs to one of $[g^{\pm M}]$.
\end{corollary}

\begin{proof}
We begin by recalling that the conjugacy classes that have translation length $0$ for any (and hence all) $d\in\Dc_\G$ are the torsion classes and that there are only finitely many of them \cite{Calegari}. Let $L,N > 0$ be as in Proposition \ref{prop.loops}. Take $K >0$ such that every conjugacy class $[g] \in \conj'$ has the property that $[g^{K}]$ has word length at least $L$. Then, by Proposition \ref{prop.loops}  $[g^{\pm KN}]$ is represented by a loop in $\mathcal{C}$. The result follows with $M = KN.$
\end{proof}


\section{Proofs of main results} \label{section.proof}
We are now ready to prove Theorem \ref{thm.1} and Theorem \ref{thm.2}. Before we prove these results we need to study the arithmetic properties of certain H\"older potentials on the Cannon coding. Throughout the rest of this section suppose we have fixed a Cannon coding for a pair $\G, S$. Let $\SS$ be the corresponding subshift which is described by the transition matrix $A$ for our fixed coding.
\begin{lemma} \label{lem.non-arith}
Suppose that $\Psi_d$ is the H\"older potential for a strongly hyperbolic metric $d \in \Dc_\G$ and that the length spectrum of $d$ is non-arithmetic. Then $\Psi_d$ is non-arithmetic on each word maximal components $\mathcal{C}$, i.e. the restrictions $\Psi_d: \Sigma_\mathcal{C} \to \R$ are non-arithmetic for any word maximal component $\mathcal{C}$. 
\end{lemma}

\begin{proof}
Suppose for contradiction that $\Psi_d$ is arithmetic on $\mathcal{C}$ and the periodic orbits take values in $a\Z$. Since $\Psi_d$ is H\"older continuous for $x = (x_n)_{n=0}^\infty$ we have
\[
S_n\Psi_d(x) = d(o,\ev_n(x)) + O(1) \text{ uniformly for $n \ge 1$}
\]
It follows that for a periodic orbit  $x = (x_0, \ldots, x_{n-1}, x_0, \ldots) \in \Sigma_{\Cc}$ with $\sigma^n(x) = x$, we have that
$
S_n\Psi_d(x) = \ell_d[\ev(x_0,\ldots, x_{n-1}, x_0)].
$
Now by Corollary \ref{cor.loops} either $\g \in \conj$ has $\ell_d[\g] = 0$ or there exist $M$ and $x = (x_0, \ldots, x_l, x_0, \ldots) \in \Sigma_\Cc$ such that $\ev(x_0,\ldots, x_l, x_0)$ belongs to $\g^{\pm M}$. In either case, $\ell_d[\g]$ belongs to $\frac{\alpha}{M} \Z$ which implies that $d$ has arithmetic length spectral contrary to our assumption.
\end{proof}

The same proof shows the following.

\begin{lemma}\label{lem.ba}
Suppose that $\Psi_d$ is the H\"older potential for a strongly hyperbolic metric $d \in \Dc_\G$ and fix a word maximal component $\Cc$. Then there exists a pair of conjugacy classes $[x], [y]$ with $\ell_d[x]/\ell_d[y]$ badly approximable if and only if there exist two periodic orbits $x_1, x_2 \in \Sigma_\Cc$ with $\sigma^{n_1}(x_1) = x_1, \sigma^{n_2}(x_2) = x_2$ such that $\Psi_d^{n_1}(x_1)/\Psi_d^{n_2}(x_2)$ is badly approximable.
\end{lemma}

We can now prove Theorem \ref{thm.2}. We start by writing $\eta$ as
\begin{equation} \label{eq.5}
\eta(s) = \sum_{n=1}^\infty \sum_{|x|_S = n} e^{-sd(o,x)}
\end{equation}
with the aim of rewriting this expression in terms of transfer operators. Let $\Cc_j$ denote the word maximal components. By Proposition \ref{prop.main} these are precisely the $-v_d\P_d$ maximal components.
We define transfer operators $L_j$, $j=1, \ldots, m$ for each component $\mathcal{C}_j$. Let $p_j$ denote the period of the component $\Cc_j$. Suppose the entire Cannon coding is described by the matrix $A$. Let, for $j=1, \ldots, m$ the matrix $C_j$ be the matrix with the same dimensions and entries as $A$ except for the fact that each word maximal component that is not $\mathcal{C}_j$ is replaced by the zero matrix. We then define transfer operators $L_{j,s}: F_\theta(\Sigma_{C_j}) \to F_\theta(\Sigma_{C_j}) $ by
\[
L_{j,s} f(x) = \sum_{\sigma(y) = x, y \neq \dot{0}} e^{-s\Psi_d(y)} f(y)
\]
for $s \in \C$. Here $\dot{0} \in \SS$ is the sequence consisting only of $0$s. Note that by construction each $C_j$ has a unique maximal component for $-v_d \Psi_d$. 

\begin{proposition}\label{prop.sr1}
For each $j=1,\ldots, m$ the operators $L_{s,j}$ satisfy the following.
\begin{enumerate}
\item Let $\chi \in F_\theta(\SS)$ be the indicator function on the one-cyclinder for the initial vertex $\ast$ in the Cannon coding. Then we have that 
\[
\sum_{|x|_S \in M_j(n)} e^{-s d(o,x)} = L_{j,s}^n \chi(\dot{0})
\]
where for $n\ge 1$,  $M_j(n) \subset \{x\in \G: |x|_S = n\}$ is the collection of group elements in $\G$ whose corresponding path in $\Gc$ starting at $\ast$ does not enter a word maximal component that is not $\Cc_j$.
\item For $s$ with $\mathrm{Re}(s) > v_d$, the spectral radius of each $L_{j,s}$ is strictly smaller than $1$.
\item There exits $\epsilon > 0$ such that for all $|s - v_d| < \epsilon$ the spectrum of each operator $L_{j,s}$ is of the following form: $L_{j,s}$ has $p_j$ simple maximal eigenvalues of the form $e^{\Pr_\Cc(-s\P_d)} e^{2\pi i l\p_j}$ and the rest of the spectrum is contained in a disk $\{|s| < \rho < 1\}$ for some $\rho$ independent of $|s-1|<\epsilon$. Here $\Pr_C(-s\P_d)$ is an analytic extension to $|s-1|<\epsilon$ of the pressure obtained from the variational principle applied to a fixed word maximal component $\Cc$. In particular $\Pr_\Cc(-s\P_d)$ is independent of $C_j$.
\end{enumerate}
\end{proposition}

\begin{proof}
The first property follows from the definitions of the objects involved and by direct calculation. The second statement follows from the fact that, since $\P_d$ is cohomologous to a strictly positive function, the maps $s \mapsto \Pr_\Cc(-s\P_d)$ are strictly decreasing for every (not necessarily maximal) connected component $\Cc$. The third property follows from Lemma 2 of \cite{pollicott.sharp}. Indeed, this result shows that the maximal part of the spectrum of each $L_{j,v_d}$ is coming from the unique word maximal component $\Cc_j$ in the definition of $C_j$. This persists under small perturbations thanks to analytic perturbation theory (upper semi-continuity of the spectrum \cite{kato}) and the conclusion follows from Proposition 4.6 of \cite{ParryPollicott}. The fact that the pressure is independent of $C_j$ follows from Lemma \ref{lem.comp1}: the pressures are determined by the Manhattan curve which is independent of $C_j$.
\end{proof}

We also need to understand the spectrum of the operators $L_{j,s}$ as $s$ varies in $\C$. The following result is key in our proof of Theorem \ref{thm.2}.

\begin{proposition}\label{prop.specna}
Suppose that $d$ has non-arithmetic length spectrum.
Let $s = v_d + it$. Then the operators $L_{s,j}$ for $j=1,\ldots,m$ do not have $1$ as an eigenvalue unless $t=0$.
\end{proposition}

\begin{proof}
By Lemma 2 of \cite{pollicott.sharp}  the spectral radius of $L_{s,j}$ is the maximum of the spectral radii over the components making up $C_j$. On the non-word maximal (or equivalently non $-v_d\P_d$ maximal components by Proposition \ref{prop.main}) the restriction of $-s\Psi_d$ has pressure strictly smaller than $1$ \cite[Theorem 4.5]{ParryPollicott}. Further if we look at the single word maximal component in $C_j$, then thanks to the assumption that $d$ has non-arithmetic length spectrum, the restriction of $\P_d$ to this component is non-arithmetic by Lemma \ref{lem.non-arith} and so by Proposition 4.5 of \cite{ParryPollicott} (see also \cite{pollicott}) $1$ can not be an eigenvalue for the operator on this component also.
\end{proof}

 Using this result we can prove Theorem \ref{thm.2}.
\begin{proof}[proof of Theorem \ref{thm.2}]
Let $\Psi_d$ be the H\"older potential encoding $d$ on a Cannon coding for $\G, S$ and assume $d$ has non-arithmetic length spectrum.  Let $\epsilon >0$ be the constant introduced in $(3)$ of Proposition \ref{prop.sr1}. Using equation (\ref{eq.5}) and part (2) of Proposition \ref{prop.sr1} we can write, for $\mathrm{Re}(s) > v_d -\epsilon$,
\begin{equation} \label{eq.2}
\eta(s) = \sum_{n=1}^\infty \sum_{j=1}^m L_{j,s}^n \chi(\dot{0}) + \beta(s)
\end{equation}
where  $\beta(s)$ is a function that is analytic on the half plane $\mathrm{Re}(s) >  v_d - \epsilon$.
To see this, note that by Lemma \ref{lem.disjointmc}  finite paths in $\Gc$ can enter at most one word maximal (or $-v_d \P_d$ maximal) component. In particular $\{x \in \G : |x|_S = n\} = \bigcup_{j=1}^m M_j(n)$ and if $x \in M_i(n) \cap M_j(n)$ for some $i,j$ then the path corresponding to $x$ must be contained in components for which the corresponding transfer operator for $-v_d \P_d$ has spectral radius strictly smaller than $1$. 
It now follows from (2) of Proposition \ref{prop.sr1} that $\eta$ is analytic and non-zero on $\mathrm{Re}(s) > v_d$.

Also, by (3) of Proposition \ref{prop.sr1} for a sufficiently small neighbourhood $U \ni\delta$ and $s \in U$ we can write
\[
L_{j,s}^n\chi(\dot{0}) = e^{n\Pr_\mathcal{C}(-s\Psi_d)} \sum_{j \in \Z/p_j\Z} e^{2\pi i ln/p_j} Q_{j,l}(s)\chi(\dot{0}) + O(\theta^n)
\]
where $0<\theta <1$ is independent of $s \in U$ and each $Q_{j,l}$ is an analytic projection valued operator on $U$ that projects a function to the simple eigenspace for the eigenvalue $e^{\text{P}_\mathcal{C}(-s\Psi_d)} e^{2\pi i l/p_j}$. It follows from this expression that if $\eta$ has a pole at $s=v_d$ then it is simple. Note that, since \[
\sum_{|x|_S= n} e^{-v_d d(o,x)}
\]
is uniformly bounded away from zero and infinity for all $n \ge 1$, $\eta$ must have a pole at $s=v_d$ and so $\eta$ has a simple pole at $s= v_d$. It is easy to check that this pole has strictly positive residue. This discussion shows there exists an extension on $\eta$ to a neighbourhood of $v_d$ that is analytic apart from a simple pole at $s = v_d$.

To prove that we can find an extension with no more poles on the line $\mathrm{Re}(s) = v_d$ it suffices, by equation (\ref{eq.2}) to show that for all $\mathrm{Re}(s) = v_d$ with $s \neq v_d$, each $L_{j,s}$ does not have $1$ as an eigenvalue. This follows from Proposition \ref{prop.specna}.
\end{proof}
The first statement in Theorem \ref{thm.1} follows immediately from Theorem \ref{thm.2} and the Wiener-Ikehara Tauberian Theorem. The second statement involving the badly approximable condition needs further work. We require more precise bounds on the spectral radii of the operators $L_{j,s}$ defined above. To obtain such bounds we use the work of Dolgopyat: it was shown in \cite{dolgopyat} that if a H\"older function over a mixing subshift of finite type satisfies some form of badly approximable condition then the corresponding classically defined transfer operator satisfies a quantifiable rate of decay/growth under iteration. This work allows us to prove the following.
\begin{proposition}\label{prop.dolg}
Let $L_{j,s}$ be the operators defined above so that $\text{P}_\Cc(-v_d\Psi_d) =0$. Suppose that there exist $[x], [y] \in \conj'$ such that $\ell_d[x]/\ell_d[y]$ is badly approximable. Then there exist constants $t_0 \ge 1$ and $ \epsilon, l, C_1, C, > 0$ such that if $s$ satisfies $\mathrm{Re}(s) > v_d - \epsilon$  and $|\mathrm{Im}(s)| \ge t_0$ and $m \ge 1$ then
\[
\|L_{j,s}^{2Nm} \| \le C_1 |\mathrm{Im}(s)| e^{2nm\text{P}_\Cc(-s\Psi_d)} \left( 1 - \frac{1}{|\mathrm{Im}(s)|^l}\right)^{m-1}
\]
where $|\mathrm{Im}(s)| \ge t_0$ and $N = \lfloor C \log|\mathrm{Im}(s)|\rfloor$.
\end{proposition}

\begin{proof}
If we additionally assume $\Sigma_{C_j}$ is a mixing subshift of finite type (i.e. $\Cc_j$ is a single strongly connected component) this result follows from the work of Dolgopyat \cite{dolgopyat} (see also \cite{pollicott.sharp.2}) and Lemma \ref{lem.ba}: indeed this lemma shows that the operator satisfies the conditions necessary to apply Dolgopyat's result. If $\Sigma_{\Cc_j}$ were transitive but not mixing the same result holds (and is easy to deduce from the mixing case). In the general case by Lemma 2 of \cite{pollicott.sharp}  each $L_{j,s}$ is quasi-compact and the isolated eigenvalues of each $L_{j,s}$ come from the components in each $C_j$. That is, the maximal part of the spectrum of $L_{j,s}$ is either coming from a non word maximal component or the word maximal component in the definition of $L_{j,s}$. The contribution to the spectrum coming from a non word maximal component is insignificant if we assume that $\epsilon >0$ is taken to be small enough: indeed, then on a non word maximal component the pressure of the restriction of $-s\Psi_d$ is strictly smaller than $0$. We therefore can apply Dolgopyats result to the maximal component in $\Cc_j$ and deduce the result.
\end{proof}
We can now prove Theorem \ref{thm.1}.
\begin{proof} [proof of Theorem \ref{thm.1}]
As mentioned above we need only prove the second statement in the theorem.
Following the same proof as Theorem \ref{thm.2} but using Proposition \ref{prop.dolg} shows that $\eta$ is analytic on a neighbourhood of the form
\[
\left\{ s : \mathrm{Re}(s) > v_d - \frac{1}{|\mathrm{Im}(s)|^\tau}, \  |\mathrm{Im}(s)| \ge t_0 \right\}
\]
for some $\tau, t_0 > 0$. More specifically we can follow the proof carried out in Section 2 of \cite{pollicott.sharp.2} that was used to study analytic extensions of zeta functions associated to hyperbolic flows: the adaptation of this proof is easy so we will not present it here. Moreover we deduce that there exist $C>0$ and $\tau'>0$ such that $|\eta(s)| \le C|\mathrm{Im}(s)|^{\tau'} $. Now that we have this description of the domain of analyticity of $\eta$ and we also know that $\eta$ has a simple pole at $v_d$ and no poles on the line $\mathrm{Re}(s) = v_d$, we can apply the classical proof from analytic number theorem to obtain the result. See  \S3 \cite{pollicott.sharp.2} or sections \S4 and \S5 \cite{cantrell.sharp} for the proof.
\end{proof}

\section{Applications: proofs}
In this section we prove the results stated in Sections 1.2 and 1.3. We begin by proving a mixing result for the Mineyev flow.
\subsection{Mixing of the Mineyev flow} \label{section.mixing}
By the coding theorem (Theorem \ref{thm.enhanced}) the Mineyev flow $\mathcal{F}_\k$ associated to a strongly hyperbolic metric $d \in \mathcal{D}_\G$ admits a coding by the suspension of a transitive subshift of finite type $\text{Sus}(\overline{\Sigma}_0,r_0)$. In particular, the subshift $\overline{\Sigma}_0$ is obtained from a maximal component $\mathcal{C}_0$ for $-v_{d}\Psi_d$. The aim of this section is to understand when the flow $(\overline{\Sigma}_0, r_0)$ is weak-mixing.

\begin{definition}
We say that $X_0=\text{Sus}(\overline{\Sigma}_0, r_0)$ is weak mixing if, when a continuous function $W : X_0 \to S^1$ and real number $a \in \R$ satisfy
\[
W \sigma_t   = e^{2\pi i a t} W
\]
then $a=0$ and $W$ is constant.
\end{definition}
When $X_0$ is not weak mixing, there exist H\"older continuous $w: \overline{\Sigma}_0 \to \C$, $a \in \R$ with $w\sigma = e^{2\pi i a r_0} w$ and so for any $x \in \overline{\Sigma}_0$ with $\sigma^{n}(x) = x$, $w(x) = w(\sigma^n(x)) = e^{2\pi i a S_nr_0(x)} w(x)$ (see \cite[Proposition 6.2]{ParryPollicott}). By Livsic's theorem this implies that $r_0$ is cohomologous to a function taking values in $\frac{1}{a} \Z$. In particular we see that, if $r_0$ is not cohomologous to a function taking values in $\frac{1}{a} \Z$, then $X_0$ is weak mixing.\\

Using this observation we can prove Theorem \ref{thm.mixing}
\begin{proof}[proof of Theorem \ref{thm.mixing}]
To prove this theorem we simply need to combine Proposition \ref{prop.loops} and Lemma \ref{lem.non-arith} with Theorem \ref{thm.enhanced}. Indeed, if $d$ has non-arithmetic length spectrum then by Lemma \ref{lem.non-arith} the function $\Psi_d$ is non-arithmetic on each word maximal (and hence each $-v_d \Psi_d$ maximal) component $\Cc$ in a fixed Cannon coding for $\G$, $S$. This implies that for each $n \ge 1$ then $n$th Birkhoff sum $S_n\Psi_d$ is also non-arithmetic since the non-lattice property is a cohomological invariant. Since the suspension flow modelling the Mineyev flow for $d$ is the suspension of $\Sigma_\C$ over a $S_N\Psi_d$ for some $N \ge 1$ we deduce that it is weakly mixing, concluding the proof.
\end{proof}

\begin{remark}
We can comment on the rate of mixing of the Mineyev flow when the base metric $d$ satisfies the bad approximability condition from Theorem \ref{thm.1}: indeed then the work of Dolgopyat \cite{dolgopyat} applies and shows that the flow is rapid mixing. We leave the details to the reader.
\end{remark}

\subsection{Correlation numbers}

In this section we prove Theorem \ref{thm.cor}. Let $d,d_\ast \in \Dc_\G$ satisfy the hypothesis of Theorem \ref{thm.cor}.
Fix a finite generating set $S$ and Cannon coding for $\G$, $S$. The proof is centered around the study of the function
\[
\o(s,t) = \exp \sum_{x\in\G} |x|_S^{-1} e^{-sd(o,x) - (\xi + it)(d_\ast(o,x) - d(o,x))}
\]
where $\xi >0$ is the unique solution to $\theta_{d_\ast/d}'(\xi) = -1$ where $\theta_{d_\ast/d}$ is the Manhattan curve for $d,d_\ast$.
In fact we can (using the Implicit Function Theorem) realise $\xi$ as the unique solution to
\[
\int_{\Sigma_\Cc} (\P_{d_\ast} - \P_d)  \ d\mu_{-\xi \P_{d_\ast} - \theta_{d_\ast/d}(\xi)\P_d} = 0
\]
for any word maximal component $\Cc$. Here $\mu_{-\xi \P_{d_\ast} - \theta_{d_\ast/d}(\xi)\P_d}$ is the equilibrium state for $-\xi \P_{d_\ast} - \theta_{d_\ast/d}(\xi)\P_d$ on $\SS_\Cc$ obtained from Proposition \ref{prop.vp}. To see this one notes that $\theta_{d_\ast/d}(t) = s$ is the solution to $\Pr_\Cc( - s\P_d - t\P_{d_\ast}) = 0$ by Theorem 5.4 of \cite{cantrell.tanaka.2} and Proposition \ref{prop.main}. In particular, $\xi \in (0,1)$ thanks to the assumption that $v_d = v_{d_\ast} =1$. To see this, note that by Theorem 1.2 in \cite{cantrell.tanaka.1} we have $\theta_{d_\ast/d}'(0) < - 1 < \theta_{d_\ast/d}'(1)$.

\begin{proposition}\label{prop.dom2}
There exists $\alpha > 0$ such that the function $\o$ satisfies:
\begin{enumerate}
\item $\o$ is analytic and non-zero in the region $\{ s: \mathrm{Re}(s) > \alpha \} \times \R$;
\item  for $t_0 \neq 0$, $\o$ is analytic and non-zero in a neighbourhood of $\{ s : \mathrm{Re}(s) \ge \alpha \} \times \{ t_0\}$;
\item  for each $t_0 \in \R$, $\o$ is analytic and non-zero in a neighbourhood of $\{s : \mathrm{Re}(s) = \alpha, \mathrm{Im}(s) \neq 0\} \times \{t_0\}$; and,
\item in a neighbourhood of $(\alpha,0)$ the function $\o(s,t)$ takes the form
\[
\frac{\phi(s,t)}{s - s(t)}
\]
where $s(t), \phi$ are analytic and non-zero. 
\end{enumerate}
\end{proposition}

\begin{proof}
Fix a Cannon coding $\Sigma_A$ for $\G, S$ where $S$ is any generating set for $\G$.
Let $\alpha$ be the unique real number such that  $(\alpha - \xi)d(o,x) + \xi d_\ast(o,x)$ has exponential growth rate $1$.
We necessarily have $\alpha -  \xi > 0$ since $\xi \in (0,1)$ and $\theta_{d_\ast/d}(\xi) = \alpha - \xi$.
In particular $d_\alpha = (\alpha - \xi)d(o,x) + \xi d_\ast(o,x)$ is a metric belonging to $\Dc_\G$ (Lemma 4.1 of \cite{oregon-reyes}) and the corresponding Busemann cocycle is a genuine cocycle obtained as a limit. This metric is represented by the H\"older continuous function $\P' = (\alpha - \xi)\P_d + \xi \P_{d_\ast} : \Sigma_A \to \R$. Following the same proof as Proposition \ref{prop.main} we see that the $\P'$-maximal components are precisely the word maximal components. 
Consider the transfer operators $L_{j,s,t} : F_\theta(\Sigma_{C_j}) \to F_\theta(\Sigma_{C_j})$ given by
\[
L_{j,s,t} f(x) = \sum_{\sigma(y) = x, y \neq \dot{0}} e^{-s\P_d(y) - (\xi + it)(\P_{d_\ast}(y) - \P_d(y))} f(y)
\]
and let $\Pr_j(s,t)$ denote the spectral radii of these operators. 
Then, $\alpha$ is the unique real number with
$\Pr_j(-\alpha \P_d - \xi(\P_{d_\ast} - \P_d)) = 0$ for each $j$. 
On the non-word maximal components, the pressures of (the restrictions of) $\P'$ are strictly smaller than $0$. We can therefore write
\begin{equation}\label{eq.bianaly}
\o(s,t) = \exp \sum_{n=1}^\infty \sum_{j=1}^m \frac{ L_{j,s,t}^n\chi(\dot{0})}{n} +\beta(s,t)
\end{equation}
where $\beta(s,t)$ is bi-analytic on $\{s: \mathrm{Re}(s) > \alpha - \epsilon\} \times \R$ for some $\epsilon >0$ (see the proof of Theorem \ref{thm.1}).
Since for each $j$,  $\Pr_j(s,t) \ge \Pr_j(s,0)$, part (1) follows.

For parts (2) and (3)  we know that $\o(s,t)$ has an analytic extension to a neighbourhood of $(\alpha+iu, t)$ as long as the function $-u\P_d - t (\P_{d_\ast} - \P_d)$ restricted to the component $C_j$ is not cohomologous to a function taking values in $2\pi\Z$. This can only happen when $u=t=0$ by Lemma \ref{lem.non-arith} since $d$ and $d_\ast$ are independent.

Part (4) requires a little more work. We need to show that for all $(s,t)$ close $(\alpha,0)$ $\Pr_j(s,t)$ is independent of $j$. To do this we note that, by Hartogs' Theorem (Theorem 1.2.5 of \cite{Hartogs}) each $\Pr_j(s,t)$ is bi-analytic in a neighbourhood of $(\alpha,0)$. Following the same proof as Lemma \ref{lem.comp1} we can show that $\Pr_j(s,t)$ is independent of $j$ when $s \in \R$ and $t \in i \R$ (i.e. $t$ is imaginary).
In particular, we consider the Manhattan curve for the pair $|\cdot|_S$ and $d_{s,t} = (s + t - \xi)d(o,x) + (\xi - t) d_\ast(o,x)$ where $S$ is a finite generating set for $\G$. Note that for all $(s,t)$ close to $(\alpha,0)$, $d_{s,t}$ remains a metric (for $s,t \in \R$).
We then consider the Poincar\'e series
\[
\sum_{x\in\G} e^{-z|x|_S - d_{s,t}(o,x)}
\]
and, by following the same proof as in Lemma \ref{lem.comp1}, see that each $\Pr_j(s,t)$ is the abscissa of converges in the variable $z$ as $s,t$ remain fixed, i.e. the $\Pr_j(s,t)$ agree for $(s,t)$ close to $(\alpha,0)$ with $s \in \R$ and $t\in i\R$.
Due to bi-analyticity and
Hartogs' Theorem this guarantees that the $\Pr_j(s,t)$ coincide on a neighbourhood of $(\alpha, 0)$. We will write $\Pr(s,t)$ instead of $\Pr_j(s,t)$ when $(s,t)$ belong to such a neighbourhood. We now use the spectral description of the operators $L_{j,s,t}$. By the above discussion there exists a neighbourhood  $U$ of $(\alpha,0)$ such that for $(s,t) \in U$ each $L_{j,s,t}$ has $p_j$ simple maximal eigenvalues of the form $e^{\Pr(s,t)}e^{2\pi i l/p_j}$ for $l=0,\ldots, p_j-1$. Using this spectral description and substituting into equation (\ref{eq.bianaly}), the claim follows.
\end{proof}

\begin{lemma}\label{lem.sym}
Let $s(t)$ denote the function from the previous proposition. Then, $\mathrm{Re}\, s(t)$ is an even function, $\mathrm{Im}\,s(t)$ is an odd function and furthermore,
\begin{enumerate}
\item $\nabla \mathrm{Re}\,s(0) = 0$;
\item $\nabla \mathrm{Im}\,s(0) = 0$;
\item $\nabla^2 \mathrm{Re}\,s(0)$ is negative definite; and,
\item $\nabla^2 \mathrm{Im}\,s(0) = 0$.
\end{enumerate}
\end{lemma}

\begin{proof}
The proof is identical to the proof of Proposition 3 in \cite{sharp} after observing that $\o(\overline{s},-t) = \overline{\o(s,t)}$.
\end{proof}

\begin{proof} [proof of Theorem \ref{thm.cor}]
We use Proposition \ref{prop.dom2}, Lemma  \ref{lem.sym} and a variation of the method presented in the work of Sharp \cite{sharp}. We will sketch the parts of the proof that follow as in \cite{sharp} and explain more carefully where extra work is needed. Let $\o(s,t)$ be as above. By Proposition \ref{prop.dom2} we can take the logarithmic derivative of $\o$ with respect to $s$. Doing so we obtain the function
\[
\sum_{x\in\G} -\frac{d(o,x)}{|x|_S} e^{-sd(o,x) - (\xi + it)(d_\ast(o,x) - d(o,x))}
\]
that has the same domain of analyticity as $\o$ and also has the property that, near $(\alpha,0)$ this function looks like
\[
-\frac{1}{s-s(t)} 
\]
plus a function that is bi-analytic. This function is of precisely the same form as the functions studied by Sharp in Sections 4 and 5 of \cite{sharp}. Following the method presented in these sections, that is, using a result of Katsuda and Sunada \cite{katsuda.sunada}, we can translate the analytic properties for our function into a counting result. Lemma \ref{lem.sym} is needed so that one can apply a refinement of the Morse lemma (see Lemma $2$ of \cite{sharp}). After fixing $\epsilon >0$ and following this argument we arrive at the asymptotic
\begin{equation}\label{eq.9}
\sum_{d(o,x) <T} \frac{d(o,x)}{|x|_S}  \chi_{[0,\epsilon]}(|d_\ast(o,x) - d(o,x)|) \sim \frac{C e^{\alpha T}}{\sqrt{T}}
\end{equation}
as $T\to\infty$ where $C> 0$ is a constant and $\chi_{[0,\epsilon]}$ represents the indicator function on $[0, \epsilon]$. We have obtained a result similar to that in the statement of Theorem \ref{thm.cor} except for the extra weighting term $d(o,x)/|x|_S$ which appeared as an artifact of our use of the Cannon coding. Such a term does not appear in \cite{sharp} because of the differences in the quantities that Sharp counts.

To conclude the proof we need to remove this weighting term. To do so we use a large deviation theorem that follows from the work of Cantrell and Tanaka \cite{cantrell.tanaka.1}. Let $d_\alpha$ be the metric in $\Dc_\G$ given by $(\alpha-\xi)d + \xi d_\ast$. Then, the same proof as that used to show Theorem 4.23 in \cite{cantrell.tanaka.1} implies that there is $\Lambda_{\alpha,S} > 0$ such that for any fixed $\epsilon'>0$
\begin{equation}\label{eq.10}
\frac{1}{\#\{x : d_\alpha(o,x) < T\}} \#\left\{ x : d_\alpha(o,x) < T, \left| \frac{d_\alpha(o,x)}{|x|_S} - \Lambda_{\alpha,S} \right| > \epsilon' \right\} 
\end{equation}
decays to $0$ exponentially quickly as $T\to\infty$. We claim that this implies that for any $\epsilon'>0$ the cardinality of 
\[
\mathcal{U}(T,\epsilon'):= \left\{ x\in\G: d(o,x) < T, |d_\ast(o,x) - d(o,x)| < \epsilon, \left| \frac{d(o,x)}{|x|_S} - \frac{\Lambda_{\alpha,S}}{\alpha} \right| > \epsilon'  \right\}
\]
grows strictly exponentially slower that $e^{\alpha T}$. To see this, note that if $x \in \mathcal{U}(T,\epsilon')$ and $|x|_S$ is sufficiently large then
\[
d_\alpha(o,x) < \alpha T + \xi\epsilon \ \ \text{ and } \ \ \left| \frac{d_\alpha(o,x)}{|x|_S} - \Lambda_{\alpha,S}\right| > \frac{\alpha\epsilon'}{2}
\]
and so (\ref{eq.10}) implies the claim. Using this claim along with equation (\ref{eq.9}) and the fact that $\epsilon'>0$ can be taken to be arbitrarily small it is easy to deduce that
\[
\sum_{d(o,x) <T} \chi_{[0,\epsilon]}(|d_\ast(o,x) - d(o,x)|) \sim \frac{C\Lambda_{\alpha,S} \ e^{\alpha T}}{\alpha \sqrt{T}}
\]
as $T\to\infty$. This concludes the proof other than showing that $0<\alpha <1$ which is shown in Corollary \ref{cor.l1} below.
\end{proof}

\begin{corollary}\label{cor.l1}
Fix $d,d_\ast \in \Dc_\G$ that are independent and have exponential growth rate $1$. Let $\xi \in \R$ be the unique solution to $\theta'_{d_\ast/d}(\xi) = -1$. Then, $\xi + \theta_{d_\ast/d}(\xi) = \alpha$ where $\alpha$ is the constant obtained from Theorem \ref{thm.cor}. Furthermore, $0 < \alpha <1$. 
\end{corollary}
\begin{proof}
The fact that $\xi + \theta_{d_\ast/d}(\xi) = \alpha$ is implicit in the proof of Proposition \ref{prop.dom2}. We therefore just need to show that $0< \alpha < 1$. It is clear that $\alpha > 0$. To see that $\alpha < 1$ note that, since $d,d_\ast$ are not roughly similar the Manhattan curve $ \theta_{d_\ast/d}$ is strictly convex \cite[Theorem 7.7]{cantrell.eduardo}. In particular, for $t\in[0,1]$, $\theta_{d_\ast/d}(t) <  t -1$ (since $\theta_{d_\ast/d}$ goes through $(0,1)$ and $(1,0)$) and so $\alpha = \xi + \theta_{d_\ast/d}(\xi) < \xi + 1 - \xi = 1$ as required.
\end{proof}
This corollary extends the main result from \cite{SharpManhattan}. 

\begin{remark}
It is likely possible to obtain the following asymptotic which is similar to the expression written in Theorem \ref{thm.cor}: for each $\epsilon > 0$ there exist $C, \alpha > 0$ such that
\[
\#\{ x \in \G: d(o,x) \in (T, T+\epsilon) , \ d_\ast(o,x) \in (T, T+\epsilon)   \} \sim \frac{Ce^{\alpha T}}{\sqrt{T}}
\]
as $T\to\infty$. To deduce this result, one could try to follow the proof of Theorem 1 from \cite{lalley} opposed to the proof of Theorem 1 from \cite{sharp} as done above. We will not pursue this here however.
\end{remark}

We conclude with a proof of Proposition  \ref{prop.ls2}.
\begin{proof}[proof of Proposition \ref{prop.ls2}]
As noted earlier, one implication is trivial. We need to show that, under the hypotheses, if $d,d_\ast$ are not roughly similar then they are independent. Suppose that $d,d_\ast$ are not roughly similar but $
a\ell_d[x] + b\ell_{d_\ast}[x] \in \Z$   for some fixed $a, b \in \R$  and for all $x\in\G$. Fix a hyperbolic element $x\in\G$.
Since the cocycle $c(x,\xi) := a\beta_{o,d}(x,\xi) + b\beta_{o,d_\ast}(x,\xi) $ is H\"older continuous and $\partial \G$ is connected we can apply Lemma 3.3 from \cite{GMM2018} to find a geodesic connecting $x^+$ and $x^-$ which is invariant under $x^n$ for some $n$. Applying the Livsic Theorem argument in Proposition 3.2 from \cite{GMM2018} we deduce that $c(x^n,x^+) = c(x^n, x^-)$ which implies that $a\ell_d[x]+ b\ell_{d_\ast}[x] = 0$. Since this is true for all hyperbolic elements $x$ (and is also trivial true for all elliptic elements) we deduce that $d,d_\ast$ have proportional marked length spectrum and so by Theorem 1.2 of \cite{cantrell.tanaka.1} they are roughly similar, a contradiction.
\end{proof}

\subsection*{Data availability}
Data sharing is not applicable to this article as no new data was created for this work.

\subsection*{Conflicts of interest}
The author states that there is no conflict of interest.

\bibliographystyle{alpha}
\bibliography{coding}

\end{document}